\documentclass[11pt]{amsart}
\usepackage{amssymb}

\numberwithin{equation}{section}
\usepackage{algorithm}
\usepackage{algorithmic}
\usepackage{amsfonts}
\usepackage{amscd}
\usepackage{color}
\usepackage{amsmath}
\usepackage{float}
\usepackage{amsthm}
\usepackage{amssymb}
\usepackage{mathrsfs}
\usepackage{amstext}
\usepackage{todonotes}
\usepackage{graphicx}
\usepackage{tikz}

\newcommand{\BIGOP}[1]{\mathop{\mathchoice%
		{\raise-0.22em\hbox{\huge $#1$}}%
		{\raise-0.05em\hbox{\Large $#1$}}{\hbox{\large $#1$}}{#1}}}

\newcommand{\BIGboxplus}{\mathop{\mathchoice%
		{\raise-0.35em\hbox{\huge $\boxplus$}}%
		{\raise-0.15em\hbox{\Large $\boxplus$}}{\hbox{\large $\boxplus$}}{\boxplus}}}
\numberwithin{equation}{section}

\theoremstyle{plain}
\newtheorem{satz}{Theorem}[section]
\newtheorem{defi}[satz]{Definition}

\newtheorem{lem}[satz]{Lemma}
\newtheorem{prop}[satz]{Proposition}
\newtheorem{rem}[satz]{Remark}

\newtheorem{exam}[satz]{Example}

\newcommand{\re}{\ensuremath{\mathbb{R}}}
\newcommand{\N}{\ensuremath{\mathbb{N}}}
\newcommand{\zz}{\ensuremath{\mathbb{Z}}}

\newcommand{\R}{\ensuremath{{\re}^d}}

\newcommand{\supp}{{\rm supp \, }}

\newcommand{\bproof}{\begin{proof}}
	\newcommand{\eproof}{\end{proof}}

\newlength{\fixboxwidth}
\newcommand{\be}{\begin{equation}}
\newcommand{\ee}{\end{equation}}
\newcommand{\beq}{\begin{eqnarray}}
\newcommand{\beqq}{\begin{eqnarray*}}
\newcommand{\eeq}{\end{eqnarray}}
\newcommand{\eeqq}{\end{eqnarray*}}

\title{}
\author{}

\begin{document}

	\title
	[A higher order  Faber spline basis for sampling discretization of functions]
	{A higher order Faber spline basis for sampling discretization of functions}

\author[Nadiia Derevianko]{Nadiia Derevianko$^{\lowercase{\rm a, b}}$}
\address{Fakult\"at f\"ur Mathematik\\ Technische Universit\"at  Chemnitz\\09107 Chemnitz, Germany, Institute of Mathematics of NAS of Ukraine, Tereshchenkivska st. 3, 01601 Kyiv-4, Ukraine}
\email{nadiia.derevianko@mathematik.tu-chemnitz.de}

\author[Tino Ullrich]{Tino Ullrich$^{\lowercase{\rm a}}$}
\address{Fakult\"at f\"ur Mathematik\\ Technische Universit\"at  Chemnitz\\09107 Chemnitz, Germany}
	\email{tino.ullrich@mathematik.tu-chemnitz.de}

\thanks{$^{\lowercase{\rm a}}$Fakult\"at f\"ur Mathematik, Technische Universit\"at  Chemnitz, 09107 Chemnitz, Germany}

\thanks{$^{\lowercase{\rm b}}$Institute of Mathematics of NAS of Ukraine, Tereshchenkivska st. 3, 01601 Kyiv-4, Ukraine}

\begin{abstract}
This paper is devoted to the question of constructing a higher order Faber spline basis for the sampling discretization of functions with higher regularity than Lipschitz. The basis constructed in this paper has similar properties as the piecewise linear classical Faber-Schauder basis \cite{Fa08} except for the compactness of the support. Although the new basis functions are supported on the real line they are very well localized (exponentially decaying) and the main parts are concentrated on a segment. This construction gives a complete  answer to Problem 3.13 in Triebel's monograph \cite{Tr2012} by extending the classical Faber basis to higher orders. Roughly, the crucial idea to obtain a higher order Faber spline basis is to apply Taylor's remainder formula to the dual Chui-Wang wavelets. As a first step we explicitly determine these dual wavelets which may be of independent interest. Using this new basis we provide sampling characterizations for Besov and Triebel-Lizorkin spaces and overcome the smoothness restriction coming from the classical piecewise linear Faber-Schauder system. This basis is unconditional and coefficient functionals are computed from discrete function values similar as for the Faber-Schauder situation.
\end{abstract}

\maketitle

\section{Introduction}

In this paper a higher order Faber spline basis for the sampling discretization of functions with higher regularity than Lipschitz is constructed. Similar as for the classical piecewise linear Faber-Schauder basis \cite{Fa08}, the coefficients in the basis expansion are computed from discrete point evaluations. The classical piecewise linear Faber-Schauder basis is nowadays a well-understood object and used in several mathematical disciplines, such as probability \cite{Cies1980}, (nonlinear) approximation and sampling recovery of multivariate functions \cite{BG2004,DD2011,Glen2018}, numerical integration and discrepancy \cite{Tr2010,Tr2012,HMOT2016}. However, from the limited regularity of the classical basis functions we may not expect approximation rates beyond $n^{-2}$ if $n$ denotes the number of degrees of freedom (e.g.\ function values). Therefore, it is a natural question to ask for a more regular variant of this basis. The basis constructed in this paper has similar properties as the classical Faber-Schauder basis except for the compactness of the support. Although the new basis functions are supported on the real line they are very well localized (exponential decay) and the main parts are concentrated on a segment, which makes them also  relevant for computational issues. This construction gives a complete  answer to Problem 3.13 in Triebel's monograph \cite{Tr2012}.

Having this new basis we consider the problem of characterizing smoothness spaces in terms of coefficients with respect to this particular basis.
This question is part of a more general problem --  the characterization of function spaces with respect to a spline (wavelet) basis system or a frame. This problem has been studied from the 1960/70ies starting with the work of Ciesielski \cite{Cies1960, Cie72}, Ropela \cite{Ro76} and Triebel \cite{Tr73,Tr78}. It has been further continued by Ciesielski and Kamont in \cite{Cies1980}, \cite{CK1995}, \cite{Kamont97} and also by Bourdaud \cite{Bou95}. In particular, the Haar system \cite{Ha10} recently attracted renewed interest, see for instance Seeger, T. Ullrich \cite{SeUl17, SeUl17_2}, Garrig{\'o}s, Seeger, T. Ullrich \cite{GaSeUl2016, GaSeUl2019} or V. Romanyuk \cite{Roman2015}, \cite{Romanyuk2016} and \cite{Roman2016}. The series of papers of D\~{u}ng  \cite{DD2009}, \cite{DD2011}, \cite{DD2013}, \cite{DD2016}, \cite{DD2018} and \cite{DD2019} introduces a multivariate ``frame-type'' spline system  (see Remark \ref{rem-ding} for more extended explanation) for this purpose. Let us also mention Schmeisser, Sickel \cite{SC2000} where a univariate Shannon sampling theory for smoothness spaces on the real line is developed. When it comes to the  multivariate (tensor product) Faber system \cite{Fa08} we refer to the recent monographs Triebel \cite{Tr2010,Tr2012}, Byrenheid \cite{Glen2018} and also to Bungartz, Griebel \cite{BG2004}.
Discretizations in terms of such non-smooth functions are also called ``non-smooth atomic decompositions'' of Besov and Sobolev spaces. We refer to the papers \cite{TW96}, \cite{Sarf2013} and \cite{SV2013} for more details in this direction.

In his 2010 monograph \cite{Tr2010} Triebel offered a new approach based on the classical Faber basis to get sampling characterizations of smoothness spaces. The result is an equivalent characterization for the norm of Besov spaces for the range of smoothness $1/p<r<\min\{1+1/p,2\}$. Note that in \cite{Tr2010} a similar question was also considered for Triebel-Lizorkin spaces with further additional restrictions on the smoothness parameter $r$ and in a more general framework for multivariate functions given on $\R$ in \cite{Glen2018}. Independently of \cite{Glen2018}, in 2011 the tensorized Faber-Schauder basis was investigated in detail for the sampling representation of Besov spaces in \cite[Section 4]{DD2011}.  The corresponding characterization uses only the values of the function at dyadic points. The lower restriction $r>1/p$ is natural and due to the availability of the point evaluation, but the upper restriction $r<\min\{1+1/p,2\}$ comes from the smoothness of Faber hat-functions. Therefore, it is a natural question to overcome this restriction and to obtain the corresponding characterization in general for all $r>1/p$.

 This question was formulated as a problem in Triebel's monograph \cite[\S 3.5]{Tr2010}.  In his books \cite[\S 3.5.2]{Tr2010} and \cite[\S 3.4]{Tr2012} he also offers some ideas how to extend results for Faber  hat-functions to higher order Faber splines. The idea was to integrate so-called higher order Battle-Lemarie spline wavelets (see \cite[Remarks 2.44 and 2.45]{Tr2010} for detailed information). Since these wavelets are not compactly supported (although exponentially decaying) and constitute an orthogonal basis in $L_2(\re)$ the coefficients in a series expansion with respect to this ``integrated'' system (see \cite{Tr2010}) can of course be represented as a linear combination of function values at dyadic points. However, the  number of values depends on the scaling level $j$. Another issue is the starting term of the expansion since the scaling function of a wavelet system can not be integrated properly (there are no vanishing moments). This is formulated as Problem 3.13 in \cite{Tr2012}.

In this paper we present the solution of this problem in the univariate setting.  We construct a higher order  B-spline basis that allows to get sampling discretizations of Besov-Triebel-Lizorkin function spaces with higher smoothness $r$. We follow the idea of constructing the Faber-Schauder basis by integrating (lifting) the Haar functions. As a replacement for Haar we use the biorthogonal wavelet system constructed by Chui and Wang \cite{Chui1992, Chuibook}, see also Lorentz, Oswald \cite{LoOsw98}. Our basis mother function is then given by the lifted dual wavelet (rather then the lifted compactly supported wavelet itself as done in \cite{Wang96, Wang95}, see also Remark \ref{rem-wang} below). It is essential that the (primal) Chui-Wang wavelet \eqref{CW01} represents a compactly supported Riesz basis in the wavelet space $W_j$ which is orthogonal with respect to different scales (see Theorem \ref{wavelet}). As a first step we therefore explicitly determine the dual Chui-Wang wavelet using tools from complex analysis.
To be more precise, we explicitly determine the coefficients $a_n^{(m)}$ in the following representation (see Theorem 2.2 for linear wavelets and Section 6 for higher order wavelets)
$$
\psi_m^*(x)=\sum\limits_{n\in \zz} a^{(m)}_n \psi_m(x-n).
$$
In \cite{Chui1992} it was shown that coefficients of this representation decay  exponentially, but as far as we know the exact formula for $a^{(m)}_n$ was not known before. Note, that there is also another construction of biorthogonal spline wavelets by Cohen, Daubechies, Feauveau \cite{CoFaDa92} based on two different multiresolution analyses. This construction has the advantage that both, primal and dual wavelet, are compactly supported. However, the semi-orthogonality property, which is crucial for our approach, is not present.

Applying Taylor's remainder formula to the dual Chui-Wang wavelets $\psi_{m}^{\ast}(\cdot)$ leads to the new basis mother functions $s_{2m}(\cdot)$. These ideas are described in Lemma \ref{recov2}. As the starting term we use the fundamental spline interpolant  $L^{2m}(\cdot)$, which, due to its fundamental interpolant property at integer points, represents the correct starting term of the expansion. Since primal Chui-Wang wavelets are compactly supported we get that the coefficients in a series expansion with respect to the higher order Faber spline basis $s_{2m;j,k}(\cdot)$ constructed in this paper can be represented as finite linear combination of samples (the number of points depends only on the order $m$ of basis functions). In case $m=2$ we obtain (see the expansion (\ref{exp-m}) and the formula for the coefficients (\ref{coef-m}) for arbitrary $m\geq 3$, $m\in \N$)
$$
	f=\sum \limits_{k \in \zz} f(k)L^4(x-k)+\sum \limits_{j =0}^{\infty} \sum \limits_{k \in \zz} \lambda_{j,k}(f) s_{j,k}(x),
$$
	where the coefficients $\lambda_{j,k}(f)$ are given as a linear combination of 4-th order differences, i.e.,
$$
	\lambda_{j,k}(f)=\dfrac{1}{6}\left(\Delta^4_{2^{-j-1}}f\left( \dfrac{2k}{2^{j+1}}\right)-4\Delta^4_{2^{-j-1}}f\left( \dfrac{2k+1}{2^{j+1}}\right)+\Delta^4_{2^{-j-1}}f\left( \dfrac{2k+2}{2^{j+1}}\right) \right)\,.
$$
Note the similarity to $\lambda_{j,k}(f)  = -\frac{1}{2} \Delta^2_{2^{-j-1}}f(2^{-j}k)$ for the classical Faber-Schauder system.

 The main discretization result of this paper is formulated in Theorem \ref{charact3}. With the use of this basis we obtain sampling discretizations of functions from Besov spaces $B^r_{p,\theta}(\mathbb{R})$ for the smoothness parameter $r$ that satisfies $1/p<r<\min\{2m-1+1/p,2m\}$ and $\max\{1/p,1/\theta\}<r<2m-1$ for Triebel-Lizorkin spaces $F^r_{p,\theta}(\mathbb{R})$.
Note that for the simplicity we consider the case of piecewise cubic splines  and in Section 6 we give the main ideas for the extension of the obtained results for Faber splines of higher order.

The Chui-Wang biorthogonal wavelet basis is of independent interest for the discretization of function spaces. It has for instance application for new characterizations in terms of ``Haar frames'', see \cite{GaSeUl19}. Therefore, we also give equivalent representations for the (quasi-)norm of Besov-Triebel-Lizorkin spaces in terms of Chui-Wang wavelet coefficients. Note also that for the periodic case very well time-localized basis functions were constructed in \cite{PS2001} and \cite{DMP2019} for one- and two-dimensional cases. The corresponding characterization of Besov spaces (for the univariate case so far) was obtained in \cite{DMP2017}.

\textbf{Outline.} This paper has the following structure. In Section 2 we give the definition of Chui-Wang wavelets and prove a formula for the explicit representation of dual linear Chui-Wang wavelets. In Section 3 we describe the construction of a piecewise cubic B-spline basis and prove uniform convergence of the expansion (i.e. in $\|\cdot\|_{\infty}$) for compactly supported continuous functions. Section 4 is dedicated to sampling characterization of Besov-Triebel-Lizorkin spaces. In Section 5 we give description of these spaces via Chui-Wang biorthogonal basis. In Section 6 we show how to extend results of Sections 2-5 to Chui-Wang wavelets and  B-splines of higher order. Finally, definitions of functions spaces and some auxiliary results we put to the Appendix.

\textbf{Notation.} As usual $\N$ denotes the natural numbers, $\N_0:=\N\cup \{0\}$, $\N_{-1}:=\N\cup \{0,-1\}$, $\zz$ and $\re$ denote the integer and real numbers respectively and let $\zz_+:=\{k\in \zz: k\geq0\}$ and $\re_+:=\{x\in \re: x\geq0\}$. For $a\in\re$ we denote by $a_+$ the number $a_+:=\max\{a,0\}$. For two nonnegative quantities $a$ and $b$ we write $a \lesssim b$ if there exists a positive constant $c$ that does not depend on one of the parameters known from the context such that $a\leq c \, b$. We write $a\asymp b$ if $a \lesssim b$ and $b \lesssim a$. We define also difference of $m$th order with a step $h \in \re$ by the formula $\Delta_h^{m} f(\cdot)=\sum\limits_{l=0}^{m} (-1)^{m-l}\binom{m}{l} f(\cdot+hl)$.
 Let further $C(\re)$ be the space of continuous functions on $\re$ with the usual supremum norm, $C_0(\re)$ be the space of compactly supported continuous functions on $\re$ and $C^r(\re)$ be the space of functions on $\re$ with continuous $r$-th derivative. By $L_p(\re)$, $0<p\leq \infty$ as usual we denote the space of Lebesgue measurable functions with the finite norm
$$
\|f\|_p:=
\begin{cases}
\Big(\int\limits_{\re} |f(x)|^p dx \Big)^{1/p}, & 0<p<\infty, \\
\mathrm{ess}\sup\limits_{x \in \re} |f(x)|, & p=\infty.
\end{cases}
$$
Let $S(\re)$ be the Schwartz space of infinitely times differentiable fast decreasing functions. By $S'(\re)$ we denote the topological dual of $S(\re)$ that is the space of tempered distributions.

\section{Non-compactly supported dual wavelets}
In this section we give definition of wavelets and B-splines (see  \cite{Chuibook}, \cite{Chui1992}) and prove an explicit representation for dual wavelets of order 2. General result for $m$-th order dual wavelets is formulated in Section 6.
\subsection{Construction of the Chui-Wang wavelets}

 Let $N_m$, $m \in \N$, be the $m$-th order B-spline with knots at $\zz$ defined by
$$
N_m(x)=(N_{m-1}*N_1)(x)=\int \limits_0^1 N_{m-1}(x-t) dt,
$$
where $N_1=\mathcal{X}_{[0,1)}$. It is clear that $\supp N_m=[0,m]$.
By $N_{m;j,k}$ we denote $N_{m;j,k}:=N_m(2^j \cdot-k)$ and let
$$
V_j=\text{span}\{N_{m;j,k}: k \in \zz \}.
$$
It is well known that spaces $V_j$ constitute a multiresolution analysis of $L_2(\re)$ (see, for example, \cite{Chui1992}), i.e. the following properties hold
\begin{itemize}
	\item [(i)] $V_j\subset V_{j+1}$ for $j \in \zz$;
	\item [(ii)] $\text{clos}_{L_2}\Big(\bigcup_{j \in \zz}V_j\Big)=L_2$;
	\item [(iii)] $\bigcap_{j \in \zz}V_j=\{0\}$;
	\item [(iv)] for each $j$ the system $\{N_{m,j,k}: \, k \in \zz \}$ is a Riesz basis in $V_j$.
\end{itemize}
Then as usually the wavelet space $W_j$ is defined as orthogonal complement of the space $V_j$ to the space $V_{j+1}$, i.e.
$$
W_j=V_{j+1}\ominus V_j, \, j \in \zz.
$$
Wavelet spaces $W_j$ are generated by some basic wavelet. Detailed information about it may be found in \cite{Chui1992} and here we only state the following result.
 \begin{satz}\label{wavelet}\cite{Chui1992}
	The $m$th order spline
	\begin{equation}\label{CW01}
	\psi_m(x)=\dfrac{1}{2^{m-1}}\sum\limits_{l=0}^{2m-2}(-1)^lN_{2m}(l+1)N_{2m}^{(m)}(2x-l),
	\end{equation}
	with support $[0,2m-1]$, is a basic wavelet that generates $W_0$, and consequently, all the wavelet spaces $W_j$, $j\in \zz$, that is $W_j=\mathrm{span}\{\psi_m(2^{j}\cdot -k), \, k \in \zz\}$.
\end{satz}
Note that $\psi_1$ is the Haar function.
\begin{figure}[h!]
	\begin{minipage}[h]{0.31\linewidth}
		\center{\includegraphics[width=1\linewidth]{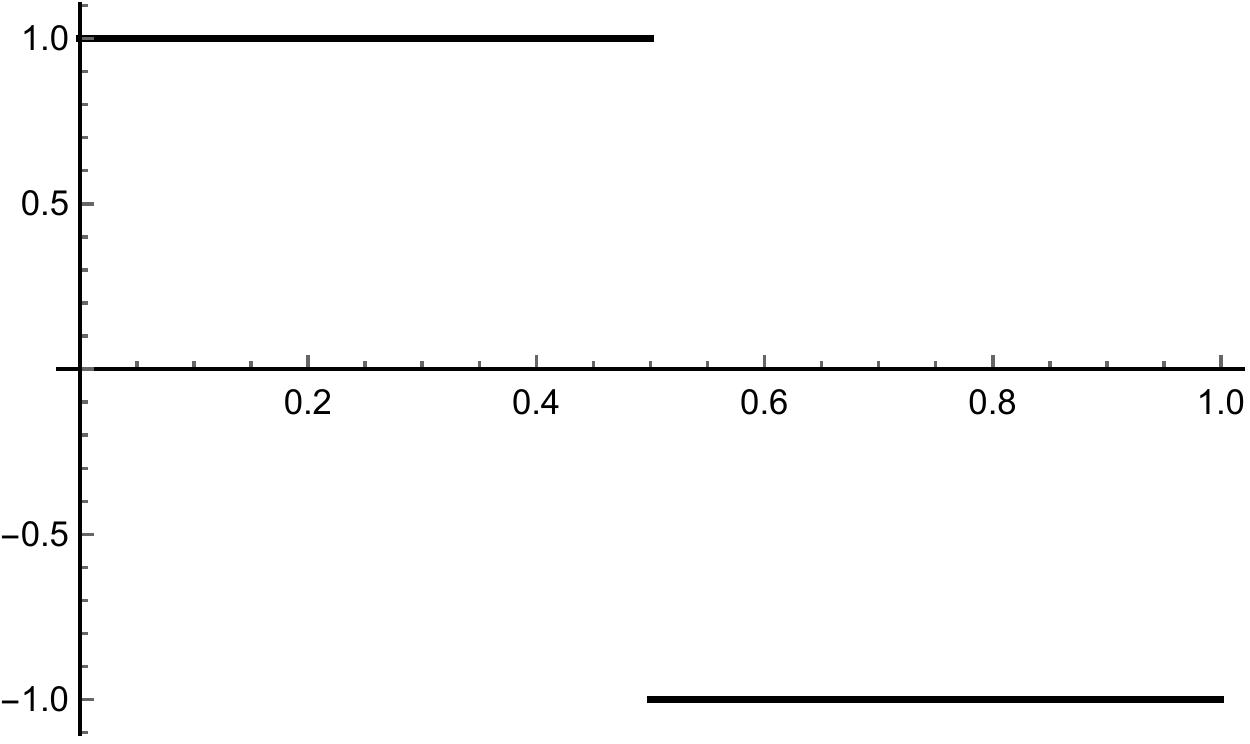}}
	\end{minipage}
	\hfill
	\begin{minipage}[h]{0.31\linewidth}
		\center{\includegraphics[width=1\linewidth]{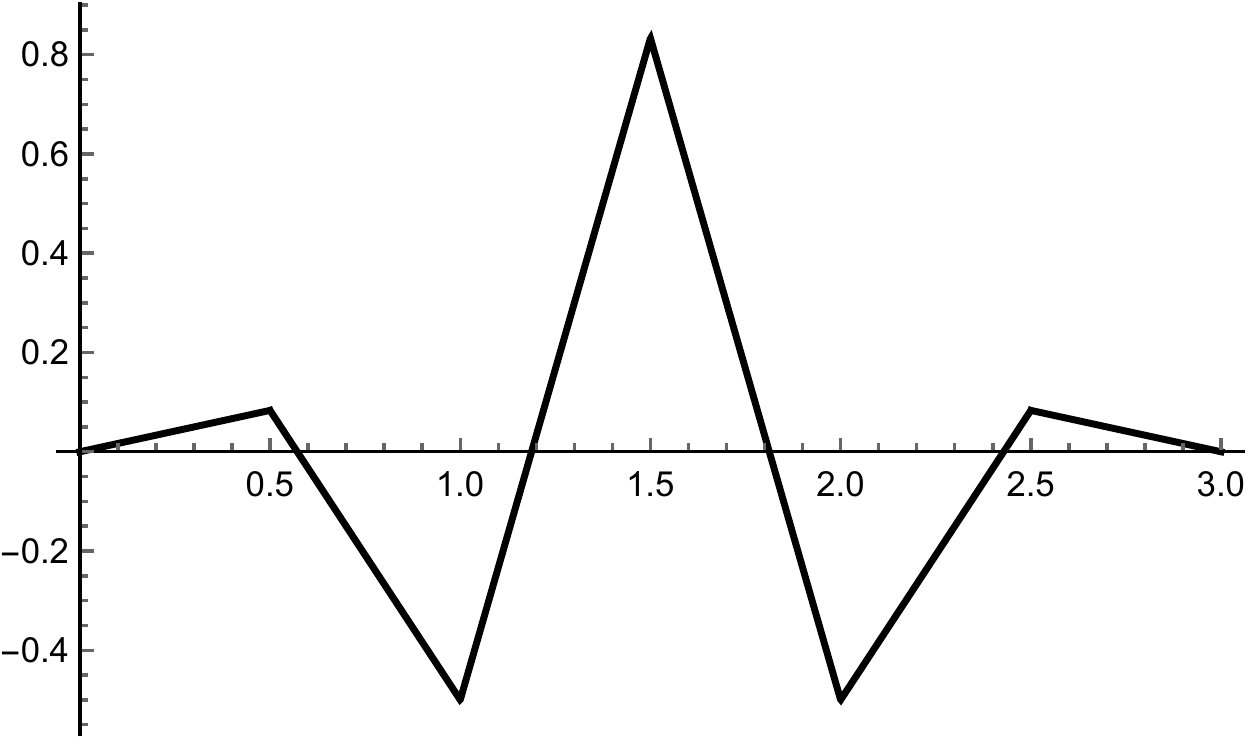}}
	\end{minipage}
	\hfill
	\begin{minipage}[h]{0.31\linewidth}
		\center{\includegraphics[width=1\linewidth]{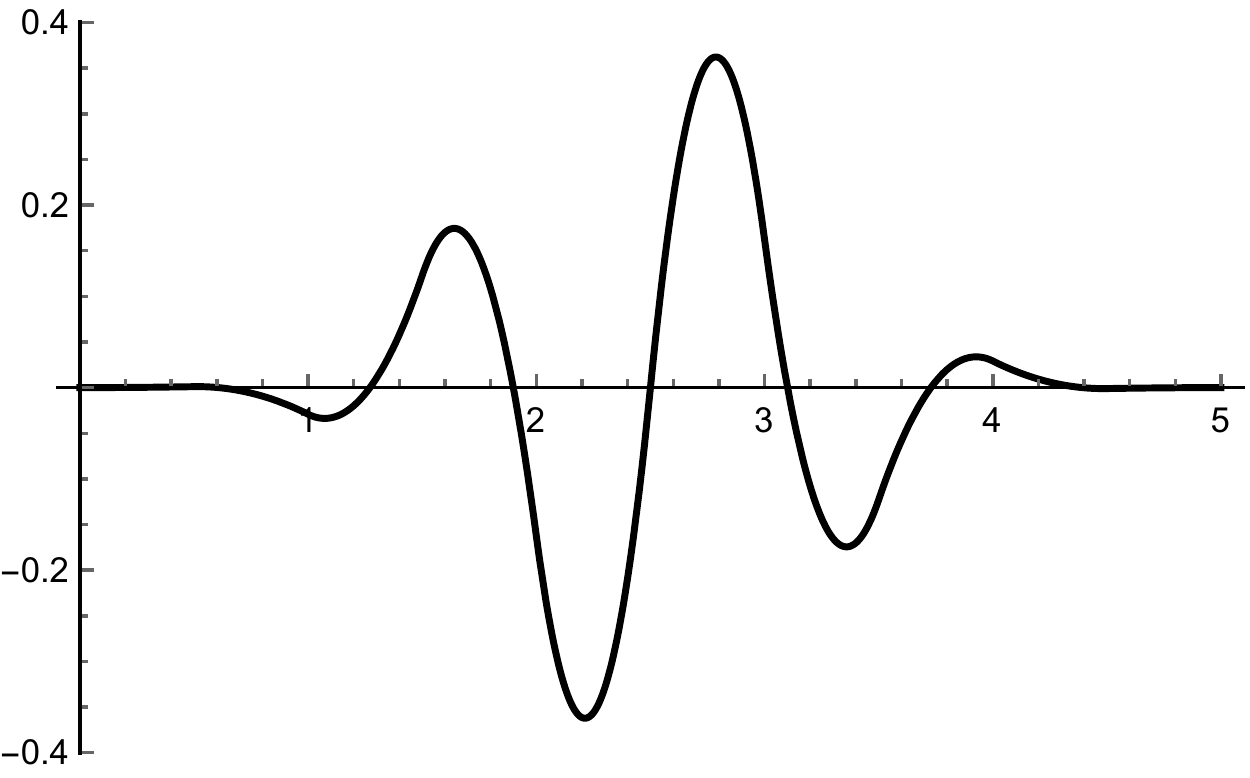}}
	\end{minipage}
  \caption{ Wavelets $\psi_1$ (left), $\psi_2$ (middle) and $\psi_3$ (right).}
	\label{wavelets}
\end{figure}

By $\psi_{m;j,k}$ we denote $\psi_{m;j,k}:=\psi_m(2^j \cdot-k)$. Due to semiorthogonality of the wavelets $\psi_{m;j,k}$  the space $L_2(\re)$ can be decomposed
$$
L_2(\re)=V_0 \oplus W_0 \oplus W_1 \oplus ... \, ,
$$
where $V_0$ is generated by $N_m(\cdot)$ and each $W_j$ by  $\psi_m(2^j \cdot)$. By $\psi_m^*$ we denote the dual wavelet of $\psi_m$, and by $N_m^*$ dual of $N_m$. Below in Theorem 2.2 we present a formula for computation of $\psi_m^*$ for $m=2$ (see Section 6 for arbitrary $m$). Then we have that  each $f \in L_2(\re)$ can be represented (see \cite{Chui1992} for details)
\begin{equation}\label{dual-repre}
f=\sum\limits_{k\in \zz} \langle f, N_{m;0,k}\rangle N_{m;0,k}^* +\sum\limits_{j\in \N_0} \sum\limits_{k\in \zz} \langle f, 2^j \psi_{m;j,k}\rangle \psi^*_{m;j,k}.
\end{equation}

\subsection{Representation of the dual wavelet}
The following theorem holds
\begin{satz} \label{dual-wav-repr}
	The dual wavelet $\psi_2^*$ can be represented as
	$$
	\psi_2^*(x)=\sum\limits_{n\in \zz} a_n \psi_2(x-n),
	$$
	where coefficients $a_n$ are defined as follows
	\begin{equation}\label{dualwavcoef}
	a_n=\begin{cases}
	(-6-4\sqrt{3})(-2-\sqrt{3})^{n-1}+(6+7\sqrt{3}/2)(7+4\sqrt{3})^{n-1} & \text{if} \ \ n\leq 1, \\
	(6-4\sqrt{3})(-2+\sqrt{3})^{n-1}+(-6+7\sqrt{3}/2)(7-4\sqrt{3})^{n-1} & \text{if} \ \ n>1.
	\end{cases}
	\end{equation}
\end{satz}
\begin{proof}
According to the definition of the dual wavelet we have $\langle \psi_2^*(\cdot-n), \psi_2(\cdot-l) \rangle=\delta_{n,l}$ what is the same as $\langle \psi_2^*(\cdot), \psi_2(\cdot-l) \rangle=\delta_{0,l}$. Then
$$
\delta_{0,l}=\langle \psi_2^*(\cdot), \psi_2(\cdot-l) \rangle=\sum \limits_{n \in \zz}a_n \langle \psi_2(\cdot-n), \psi_2(\cdot-l) \rangle=\sum \limits_{n \in \zz}a_n \langle \psi_2(\cdot+l-n), \psi_2(\cdot) \rangle=\sum \limits_{n \in \zz}a_n c_{l-n},
$$
where $c_{l-n}=\langle \psi_2(\cdot+l-n), \psi_2(\cdot) \rangle$.

Let us further find $a_n$ from the condition $\sum_{n \in \zz}a_n c_{l-n} = \delta_{0,l}$. Computing the coefficients $c_l$ for $l\in \zz$ we get $c_0=\frac{1}{4}$, $c_{\pm1} =\frac{5}{108}$, $c_{\pm 2}=-\frac{1}{216}$ and $c_l=0$ for $|l|\geq3$.
Then we consider  the product of the following two polynomials
$$
t_a=\sum \limits_{n \in \zz} a_n \mathrm{e}^{inx} \ \text{ and } \ t_c=\sum \limits_{l \in \zz} c_l \mathrm{e}^{ilx}.
$$
We have
$$
(t_a\cdot t_c)(x)=\sum \limits_{n \in \zz} \sum \limits_{l \in \zz} a_n c_l \mathrm{e}^{i(n+l)x}=\sum \limits_{l \in \zz} \left( \sum \limits_{n \in \zz} a_n c_{l-n} \right) \mathrm{e}^{ilx}.
$$
Since the sum in the brackets $\sum \limits_{n \in \zz} a_n c_{l-n}=\delta_{l,0}$ we get
$$
(t_a\cdot t_c)(x)=\sum \limits_{l \in \zz} \delta_{l,0} \mathrm{e}^{ilx}=1.
$$
Then $t_a(x)=\frac{1}{t_c(x)}$ or
$$
\sum \limits_{n \in \zz} a_n \mathrm{e}^{inx} = \frac{1}{t_c(x)}.
$$
Multiplying both parts of the last equality by $\mathrm{e}^{-inx}$ and integrating over $[0,2\pi]$, we obtain
\begin{align*}
a_n=&\frac{1}{2\pi} \int \limits_{0}^{2\pi} \frac{\mathrm{e}^{-inx}}{-\frac{1}{216}\mathrm{e}^{-2ix}+\frac{5}{108}\mathrm{e}^{-ix}+\frac{1}{4}+\frac{5}{108}\mathrm{e}^{ix}-\frac{1}{216}\mathrm{e}^{2ix}} \, dx  \\
=&\frac{216}{2\pi i} \int \limits_{0}^{2\pi} \frac{i \mathrm{e}^{ix} \mathrm{e}^{-i(n-1)x}}{-1+10\mathrm{e}^{ix}+54\mathrm{e}^{2ix}+10\mathrm{e}^{3ix}-\mathrm{e}^{4ix}} \, dx.  \\
\end{align*}
Making changes of variables $ \mathrm{e}^{ix}=z$ we have
\begin{equation} \label{ak}
a_n=-\frac{216}{2\pi i} \int \limits_{|z|=1} \frac{z^{-n+1}}{1-10z-54z^2-10z^3+z^4}\, dz.
\end{equation}

\begin{figure}[h!]
		\center{\includegraphics[width=0.5\linewidth]{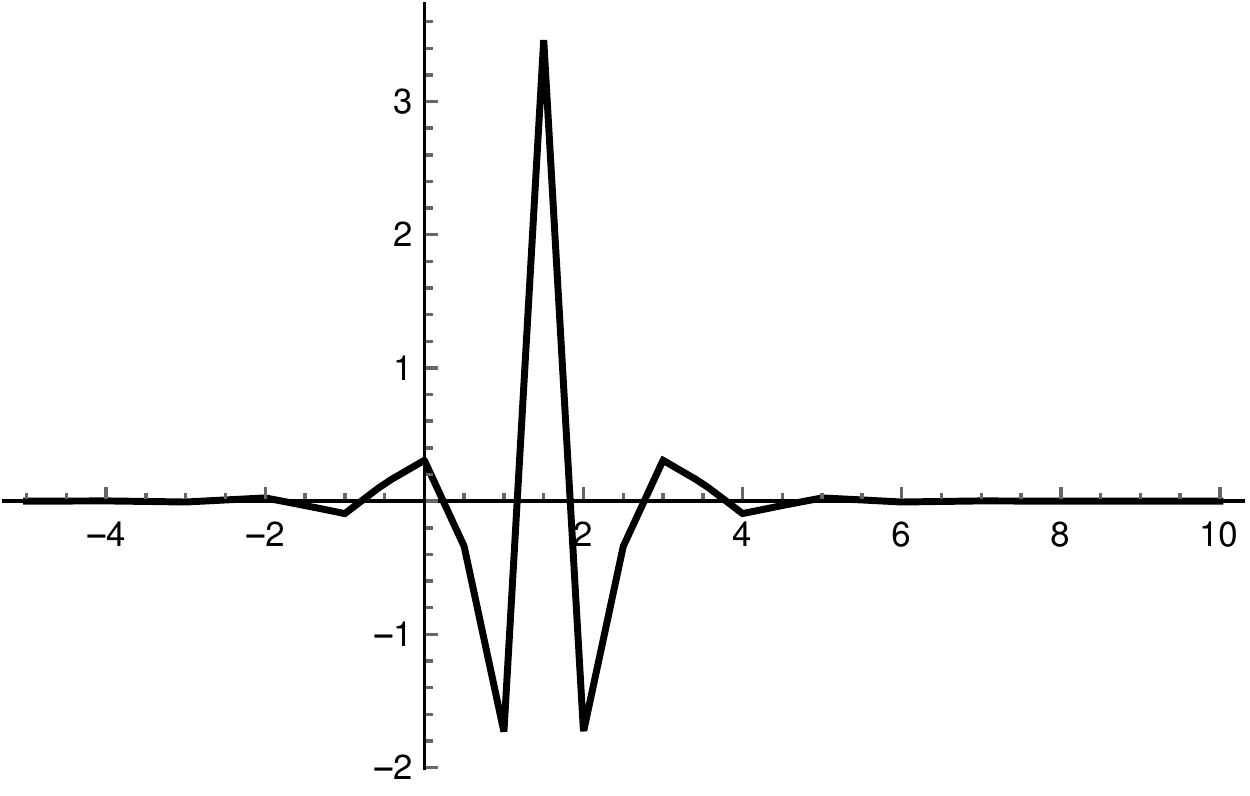}}
	\caption{The dual wavelet $\psi_2^*$}
	\label{example3-1}
\end{figure}

Let us further consider the following function $q(z)=\frac{z^{-n+1}}{1-10z-54z^2-10z^3+z^4}$.  The numbers $z_0=7-4\sqrt{3}$, $z_1=-2-\sqrt{3}$, $z_2=-2+\sqrt{3}$ and $z_3=7+4\sqrt{3}$ are roots of the denominator.

Further we consider two cases.

1) $n>1$:
$$
\frac{1}{z^{n-1}}\frac{1}{1-10z-54z^2-10z^3+z^4}=\sum \limits_{l=1}^{n-1}\frac{A_l}{z^l}+\frac{B_0}{z-z_0}+\frac{B_1}{z-z_1}+\frac{B_2}{z-z_2}+\frac{B_3}{z-z_3},
$$
where $A_l$ and $B_s$ are some constants. Multiplying both parts of the last equality by $z-z_s$, $s=0,1,2,3$, and then putting $z=z_s$ we find that
\begin{equation} \label{bs}
B_s=\frac{1}{z_s^{n-1}\prod\limits_{i=0,1,2,3, i\neq s}(z_s-z_i)}.
\end{equation}

It is easy to see that $A_1=-B_0-B_1-B_2-B_3$, $A_l=0$, $l=2,3,...,n-1$. Then using (\ref{ak}) and Cauchy's integral formula we get that
$$
a_n=-216 \left(\frac{1}{2\pi i} \int \limits_{|z|=1} \dfrac{A_1}{z} \, dz + \sum \limits_{s=0}^3 \frac{1}{2\pi i} \int \limits_{|z|=1} \dfrac{B_s}{z-z_s} \, dz \right) =-216(A_1+B_0+B_2)=216(B_1+B_3).
$$
Using formula for $B_s$ (\ref{bs}) we have
$$
a_n=(6-4\sqrt{3})(-2+\sqrt{3})^{n-1}+(-6+7\sqrt{3}/2)(7-4\sqrt{3})^{n-1}, \, n>1.
$$

2) $n\leq1$: here we use notation
$$
B_s=\frac{1}{\prod\limits_{i=0,1,2,3, i\neq s}(z_s-z_i)}.
$$
From (\ref{ak}) and Cauchy's integral formula we obtain
$$
a_n=-216\left(\sum \limits_{s=0}^3 \frac{1}{2\pi i} \int \limits_{|z|=1} \dfrac{B_sz^{-n+1}}{z-z_s} \, dz \right)=-216\left(B_0 z_0^{-n+1}+B_2z_2^{-n+1}\right).
$$
Substituting $B_s$ in the last formula we get that
$$
a_n=(-6-4\sqrt{3})(-2-\sqrt{3})^{n-1}+(6+7\sqrt{3}/2)(7+4\sqrt{3})^{n-1}.
$$
\end{proof}

By using similar technique as in the proof of Theorem \ref{dual-wav-repr} it is easy to show that for $N_2^*$,
$$
N_2^*(x)=\sum\limits_{n\in \zz} b_n N_2(x+1-n),
$$
where $b_n=(-1)^n \sqrt{3} (2-\sqrt{3})^{|n|}$.

Note that the coefficients $a_n$ from Theorem \ref{dual-wav-repr} and $b_n$ from representation for $N_2^*$ decay exponentially with respect to $|n|$, what is crucial for the construction of cubic Faber splines.

\section{Construction of piecewise cubic Faber splines in the space of compactly supported continuous functions}
In this section we present a construction of piecewise cubic Faber spline basis and prove the uniform convergence in the space $C_0(\re)$.

First we give some necessary definitions. We define the cardinal spline function as (see \cite{Chuibook} for details)
\begin{equation}\label{card-sp}
L^m(x):=\sum \limits_{n \in \zz} c_n^{(m)}N_m(x+m/2-n),
\end{equation}
with the property $L^m(j)=\delta_{j,0}$, $j \in \zz$.
By $J^mf(x)$ we define the following interpolation polynomial
\begin{equation}\label{intpol}
(J^mf)(x):=\sum \limits_{n \in \zz} f(n)L^m(x-n).
\end{equation}
It is clear that $J^mf(j)=f(j)$ for $j \in \zz$.

Further we are interested in the case $m=4$ (see \cite[P. 112]{Chuibook})
$$
L^4(x)=\sum \limits_{n \in \zz} (-1)^n \sqrt{3} (2-\sqrt{3})^{|n|} N_4(x+2-n).
$$
For $N \in \N$ we define the scaling version of the operator $J^4$ as
\begin{equation}\label{oper-j4}
(J^4_Nf)(x)=\sum \limits_{n \in \zz} f(2^{-N} n) L^4(2^N x-n),
\end{equation}
and then $(J^4_Nf)(n/2^N)=f(n/2^N)$ for $n\in \zz$. By $V_N^4$ we denote the space
$$
V_N^4:= \left\{ f: \, f:=\sum \limits_{n \in \zz} c_n N_4(2^N\cdot-n), \, \{c_n\}_{n \in \zz} \in \ell_1 \right\}.
$$
Due to the compactness of the support of $N_4$ we have $V_N^4\subset L_1(\re)$.	
\subsection{Construction of piecewise cubic Faber splines} In this subsection we show the main ideas of the construction of piecewise cubic Faber splines.

\begin{lem} \label{recov2}
	Every $f \in V_N^4$ can be reproduced by the operator $S_N$, i.e. $(S_Nf)(x)=f(x)$, $\forall x \in \re$, where $S_N$ is defined as follows
	\begin{equation} \label{expansion1}
	S_Nf(x)=\sum \limits_{k \in \zz} f(k)L^4(x-k)+\sum \limits_{j =0}^{N-1} \sum \limits_{k \in \zz} \lambda_{j,k}(f) s_{j,k}(x),
	\end{equation}
	where  coefficients $\lambda_{j,k}(f)$ are defined as
	\begin{equation} \label{coeflin1}
	\lambda_{j,k}(f)=\dfrac{1}{6}\left(\Delta^4_{2^{-j-1}}f\left( \dfrac{2k}{2^{j+1}}\right)-4\Delta^4_{2^{-j-1}}f\left( \dfrac{2k+1}{2^{j+1}}\right)+\Delta^4_{2^{-j-1}}f\left( \dfrac{2k+2}{2^{j+1}}\right) \right),
	\end{equation}
	and piecewise cubic polynomials $s_{j,k}$ are defined as follows
	\begin{equation} \label{bspline}
	s_{j,k}(x)= \sum \limits_{n \in \zz} a_n  v(2^jx-k-n),
	\end{equation}
	here coefficients $a_n$ are defined by (\ref{dualwavcoef}) and
\begin{equation} \label{bfun}
	v(t)=\frac{1}{36}
	\begin{cases}
	t^3, & 0\leq t \leq 1/2, \\
	1-6t+12t^2-7t^3, & 1/2<t \leq 1, \\
	-22+63t-57t^2+16t^3, & 1<t\leq 3/2, \\
	86-153t+87t^2-16t^3, & 3/2<t\leq 2, \\
	-98+123t-51t^2+7t^3, & 2<t\leq 5/2, \\
	27-27t+9t^2-t^3, & 5/2<t\leq 3,\\
	0,&\text{otherwise}.	
	\end{cases}
	\end{equation}
\end{lem}	
\begin{proof}
	Since $f\in C^2(\re)$ (because $f \in V_N^4$), we can consider $f^{(2)}=(J^4f)^{(2)}+(f-J^4f)^{(2)}$. According to definition of the space $V_N^4$ we get that $f^{(2)} \in L_2(\re)$ and then from the viewpoint on (\ref{dual-repre}) we have the following expansion (in the sense of $L_2$ convergence):
	\begin{align*}
	f^{(2)}=& \sum \limits_{k \in \zz} \langle f^{(2)}, N_{2;0,k}\rangle N_{2;0,k}^*+\sum \limits_{j \in \zz_+} \sum \limits_{k \in \zz} \langle f^{(2)},2^j\psi_{2;j,k} \rangle \psi_{2;j,k}^*\\
	=&\sum \limits_{k \in \zz} \langle (J^4f)^{(2)}, N_{2;0,k}\rangle N_{2;0,k}^*+\sum \limits_{j \in \zz_+} \sum \limits_{k \in \zz} \langle (J^4f)^{(2)},2^j\psi_{2;j,k} \rangle \psi_{2;j,k}^*\\
	&+\sum \limits_{k \in \zz} \langle (f-J^4f)^{(2)}, N_{2;0,k}\rangle N_{2;0,k}^*+\sum \limits_{j \in \zz_+} \sum \limits_{k \in \zz} \langle (f-J^4f)^{(2)},2^j\psi_{2;j,k} \rangle \psi_{2;j,k}^*.
	\end{align*}

\begin{figure}[h!]
	\begin{minipage}[h]{0.45\linewidth}
		\center{\includegraphics[width=1\linewidth]{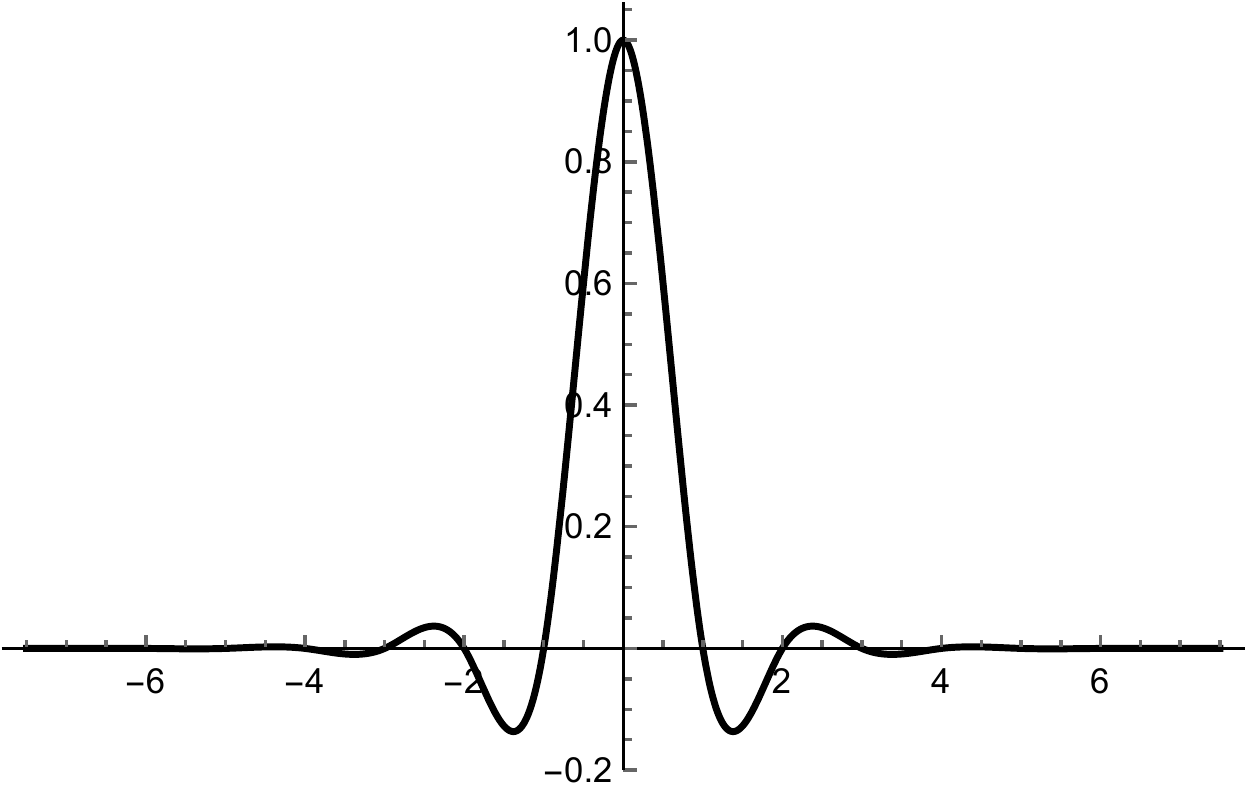}}
	\end{minipage}
	\hfill
	\begin{minipage}[h]{0.45\linewidth}
		\center{\includegraphics[width=1\linewidth]{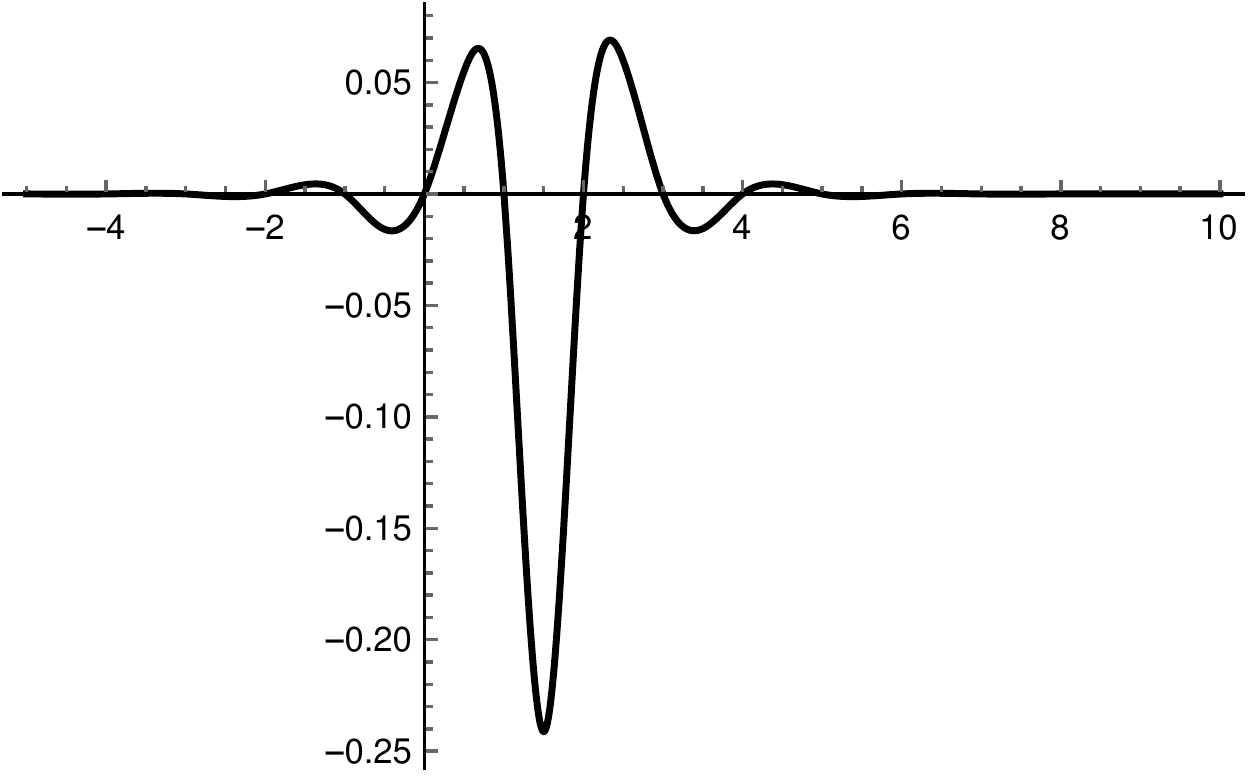}}
	\end{minipage}
	\caption{ The piecewise cubic basis functions $L^4$ (left) and $s_{0,0}$ (right)}
	\label{bas-f}
\end{figure}

	Since $(J^4f)^{(2)}$ can be represented as
	$$
	(J^4f)^{(2)}(x)=\sum \limits_{k \in \zz} \sum \limits_{j \in \zz} f(k) (-1)^{j}(2-\sqrt{3})^{|j|} \sum \limits_{m=0}^2 (-1)^m \binom{2}{m}N_2(x+2-j-k-m),
	$$
	then
	$(J^4f)^{(2)} \in V_0$ and consequently  $\sum \limits_{k \in \zz} \langle (J^4f)^{(2)}, N_{2;0,k}\rangle N_{2;0,k}^*=(J^4f)^{(2)}$.
	Further we will often use
	\begin{equation}\label{dspline}
	N_m^{(m)}(x)=\sum \limits_{j=0}^m (-1)^j \binom{m}{j} \delta(x-j),
	\end{equation}
were $\delta$ denotes  the Dirac delta distribution.

	Let us consider the following coefficients $\langle g^{(2)}, N_{2;0,k}\rangle$ for some $g \in C^2$:
	\begin{align*}
	\langle g^{(2)}, N_{2;0,k}\rangle&=\int \limits_{-\infty}^{\infty} g^{(2)}(t)N_2(t-k) \, dt =\int \limits_{-\infty}^{\infty} g(t)N_2^{(2)}(t-k) \, dt \\
	&= \sum \limits_{m=0}^2 (-1)^m \binom{2}{m} \int \limits_{-\infty}^\infty g(t)\delta(t-k-m) \, dt=\sum \limits_{m=0}^2 (-1)^m \binom{2}{m} g(k+m).
	\end{align*}
	According to this $\langle (f-J^4f)^{(2)}, N_{2;0,k}\rangle=\sum \limits_{m=0}^2 (-1)^m \binom{2}{m} (f-J^4f)(k+m)=0$ since $J^4f(j)=f(j)$ for $j \in \zz$. Then
	\begin{equation} \label{dexpansion}
	f^{(2)}=(J^4f)^{(2)}+\sum \limits_{j \in \zz_+} \sum \limits_{k \in \zz} \langle f^{(2)},2^j\psi_{2;j,k} \rangle \psi_{2;j,k}^*.
	\end{equation}
	Further we find the coefficients $\langle f^{(2)},2^j\psi_{2;j,k} \rangle$. From Theorem \ref{wavelet} and formula (\ref{dspline}) we get
	\begin{align*}
	\psi_{2;j,k}^{(2)}(x)&=2^{2j+1}\sum\limits_{l=0}^2 (-1)^l N_4(l+1) N_4^{(4)}(2^{j+1}x-2k-l)  \\
	&=2^{2j+1}\sum\limits_{l=0}^2 (-1)^l N_4(l+1)\sum \limits_{m=0}^4 (-1)^m \binom{4}{m}\delta(2^{j+1}x-2k-l-m).
	\end{align*}
	Therefore we can write for $\langle f^{(2)},2^j \psi_{2;j,k} \rangle$
	\begin{align}
	\langle f^{(2)},2^j\psi_{2;j,k} \rangle &= 2^j\int \limits_{-\infty}^\infty  f^{(2)}(t) \psi_{2;j,k}(t) \, dt = 2^j\int \limits_{-\infty}^\infty  f(t) \psi^{(2)}_{2;j,k}(t) \, dt \notag \\
	&=2^{3j+1} \sum\limits_{l=0}^2 (-1)^l N_4(l+1)\sum \limits_{m=0}^4 (-1)^m \binom{4}{m}  \int \limits_{-\infty}^\infty f(t) \delta(2^{j+1}t-2k-l-m) \, dt \notag \\
	&=2^{3j+1} \sum\limits_{l=0}^2 (-1)^l N_4(l+1)\sum \limits_{m=0}^4 (-1)^m \binom{4}{m} \dfrac{1}{2^{j+1}} \int \limits_{-\infty}^\infty f(t) \delta\left(t-\dfrac{2k+l+m}{2^{j+1}}\right) \, dt \notag \\
	&=2^{2j}\sum\limits_{l=0}^2 (-1)^l N_4(l+1)\sum \limits_{m=0}^4 (-1)^m \binom{4}{m} f \left(\dfrac{2k+l+m}{2^{j+1}} \right) \notag \\
	&=2^{2j}\sum\limits_{l=0}^2 (-1)^l N_4(l+1) \Delta^4_{2^{-j-1}}f\left( \dfrac{2k+l}{2^{j+1}}\right).  \label{coef}
	\end{align}
	
	Let us further denote $\lambda_{j,k}(f)$ as in (\ref{coeflin1}).
	From (\ref{intpol}), (\ref{coeflin1}), (\ref{dexpansion}) and (\ref{coef}) we have the following expansion for the second derivative
	$$
	f^{(2)}=(J^4f)^{(2)}+\sum \limits_{j \in \zz_+ } \sum \limits_{k \in \zz} 2^{2j} \lambda_{j,k}(f) \psi_{2;j,k}^*.
	$$ 	

	Since $f \in V_N^4$ we have that $\langle f^{(2)},2^j \psi_{2;j,k}\rangle=0$ for $j\geq N$. Therefore,
	\begin{equation} \label{ddexpansion}
	f^{(2)}=(J^4f)^{(2)}+\sum \limits_{j=0}^{N-1} \sum \limits_{k \in \zz} 2^{2j} \lambda_{j,k}(f) \psi_{2;j,k}^*.
	\end{equation}  	
	By using Taylor expansion we get
	$$
	f(x)=f(x_0)+f'(x_0)(x-x_0)+\int\limits_{x_0}^x\frac{f^{(2)}(t)}{1!}(x-t) \, dt.
	$$
	When $x_0\rightarrow -\infty$ we get (because of $\lim \limits_{|x|\rightarrow \infty}f(x)=0$ and  $\lim \limits_{|x|\rightarrow \infty}f'(x)=0$)
	\begin{equation} \label{taylor-rem}
	f(x)=\int \limits_{-\infty}^x \frac{f^{(2)}(t)}{1!}(x-t) \, dt.
	\end{equation}
	Note, that
		\begin{equation} \label{bas-fun}
		\int\limits_{-\infty}^{x} \psi^*_{2;j,k}(t)(x-t) \, dt =\sum\limits_{n \in \zz} a_n \int\limits_{-\infty}^{x} \psi_2(2^j t -k-n) (x-t) \, dt =\frac{1}{2^{2j}} \sum\limits_{n \in \zz} a_n v(2^j x-k-n),
	\end{equation}
	where $v$ is defined by (\ref{bfun}).

Applying Taylor expansion to (\ref{ddexpansion}) and taking (\ref{taylor-rem}) and (\ref{bas-fun}) into account, we get the result.
\end{proof}

\subsection{Uniform convergence in the space $C_0(\re)$} Before formulating the main result of this subsection we prove some auxiliaries statements.
The following lemma is probably known, but since we were unable to find a reference we give a proof.
\begin{lem} \label{recov1}
	Every $f \in V_N^4$ can be reproduced by the fundamental spline interpolation operator $J_N^4$, i.e.
\begin{equation} \label{j4}
(J_N^4f)(x)=f(x), \ \forall x \in \re.
\end{equation}
\end{lem}
\begin{proof}
	We prove that $J^4(N_4)=N_4$. According to definition of $J^4$ we have
	\begin{align*}
	(J^4N_4)(x)&=\sum \limits_{k \in \zz} N_4(k)L^4(x-k)= \sum \limits_{k \in \zz} N_4(k) \sum \limits_{l \in \zz} c_l^{(4)} N_4(x+2-l-k) \\
	&= \sum \limits_{n \in \zz} \sum \limits_{l \in \zz}  N_4(n+2-l) c_l^{(4)} N_4(x-n)=\sum \limits_{n \in \zz} N_4(x-n) \sum \limits_{l \in \zz} c_l^{(4)}  N_4(n+2-l) \\
	&=\sum \limits_{n \in \zz} N_4(x-n)  L^4(n)=\sum \limits_{n \in \zz} N_4(x-n) \delta_{n,0}=N_4(x).
	\end{align*}
Now (\ref{j4}) is a trivial consequence of the last equality.
\end{proof}

In the next lemma we consider functions $f \in C_0(\re)$, since then all sums in (\ref{expansion1}) is finite and $S_N f$ is well defined.

\begin{lem} \label{recov3}
	For every function $f \in C_0(\re)$ we have $S_Nf\equiv J_N^4f$, where $S_N$ and $J_N^4$ are defined by (\ref{expansion1}) and (\ref{oper-j4}) respectively.
\end{lem}
\begin{proof}
	Let $f \in C_0(\re)$ and $x \in \supp f$. Since $J_N^4f \in V_N^4$ according to Lemma \ref{recov2} we have that $S_N(J_N^4f)(x)=J_N^4f(x)$. On the other hand, since $J_N^4f(k/2^N)=f(k/2^N)$, $k \in \zz \cap \supp f$, according to definition of $S_N$ we get that $S_N(J_N^4f)(x)=S_Nf(x)$.  	
\end{proof}	

\begin{rem}\label{sn-int}
As a consequence from the last lemma we have that the operator $S_N$ interpolates a function $f \in C_0(\re)$ at points $k/2^N$, $k \in \zz$, i.e.
$$
S_Nf(k/2^N)=f(k/2^N), \ \ k \in \zz.
$$	
\end{rem}

\begin{satz}\label{conv-c}
	For a  function $f \in C_0(\re)$, we have that
	\begin{equation}\label{C-conv}
	\lim \limits_{N\rightarrow \infty} \|f - S_Nf\|_{\infty}=0,
	\end{equation}
	where $S_N$ is defined by (\ref{expansion1}).
\end{satz}
\begin{proof}
	By $Q_N^m$ we denote a quasi-interpolation operator
	$$
	Q_{N}^m(x)=\sum\limits_{k \in \nu(N)} a_k^{(N)}(f) N_m(2^N x-k),
	$$
	where the set $\nu(N)$ is  finite and the functional $a_k^{(N)}(f)$ is defined by using finite number of function values (see \cite{DD2016} for details).
	Since according to Lemma \ref{recov3} we have that $S_N=J_N^4$, we can write
	\begin{align*}
	\|f - S_Nf\|_{\infty}&=\|f - J_N^4f\|_{\infty}\leq \|f - Q_{N}^4f\|_{\infty}+ \| Q_{N}^4f-J_N^4f\|_{\infty}\\
	&=\|f - Q_{N}^4f\|_{\infty} + \|J_N^4\left(f - Q_{N}^4f \right)\|_{\infty} \leq \left(1+ \|J_N^4\|_{\infty \rightarrow \infty}\right) \|f - Q_{N}^4f\|_{\infty}.
	\end{align*}
	We used that $Q_{N}^4f \in V_N^4$ and according to Lemma \ref{recov1} $J_N^4(Q_{N}^4f)=Q_{N}^4f$.
	
	Further we use the facts that $\|f - Q_{N}^4f\|_{\infty} \rightarrow 0$ if $N\rightarrow \infty$ (see \cite{DD2016}) and the norm $\|J_N^4\|_{\infty \rightarrow \infty}$ is bounded (see \cite{Rich75}). It implies  (\ref{C-conv}).
\end{proof}

\begin{rem}\label{rem-wang}
We would like to point out that the paper \cite{Wang96} is rather close to us. However, a significant difference is the fact, that our basis is the integrated (lifted) dual Chui-Wang wavelet. This dual wavelet had to be determined explicitly in a first step. It is not compactly supported however very well localized (exponentially decaying). Similar as in \cite{Wang96, Wang95} the basis coefficients of the Faber spline basis are computed from a linear combination of discrete function values. In contrast to \cite{Wang96} these coefficients are explicitly determined and use only local discrete  information of the function $f$. This is in analogy to the classical Faber-Schauder basis which is the reason why we call it Faber spline.
\end{rem}

Further for the convenience we use the following notation. We define sequence of coefficients $\lambda_{j,k}(f)$ and functions $s_{j,k}$ for $j\in \N_{-1}$, $k \in \zz$, in the following way: if $j\geq 0$ we use definition (\ref{coeflin1}) and (\ref{bspline}) consequently, and if $j=-1$ we put
$
\lambda_{-1,k}(f):=f(k)
$
and
\begin{equation}\label{b2}
s_{-1,k}(x):=L^4(x-k).
\end{equation}
Then from Theorem \ref{conv-c} for each function $f\in C_0(\re)$ we have the following expansion
\begin{equation}\label{repr-b}
f(x)=\sum \limits_{j \in \N_{-1}} \sum \limits_{k \in \zz} \lambda_{j,k}(f) s_{j,k}(x),
\end{equation}
where convergence is understood in the sense of the space $C$.
	
Further we prove the uniqueness of this expansion. We  show that all coefficients $c_{j,k}$ in the following expansion
	\begin{equation} \label{expansion-f-un}
0=\sum \limits_{k \in \zz} c_{-1,k}L^4(x-k)+\sum \limits_{j =0}^{\infty} \sum \limits_{k \in \zz} c_{j,k} s_{j,k}(x),
\end{equation}
 equal to $0$. Note that from definition of piecewise cubic functions $s_{j,k}$ we have that  $s_{j,k}(n)=0$ for $n \in \zz$. If we put $x=n$, $n \in \zz$, in (\ref{expansion-f-un}) we get $c_{-1,k}=0$ for all $k \in \zz$. Then we apply the functional $\lambda_{l,n}(f)$ with $l \in \N_0$ and $n \in \zz$ to the series (\ref{expansion-f-un}). Since
 $$
 \lambda_{l,n}(s_{j,k})= \langle s_{j,k}^{(2)}, 2^m \psi_{2;l,n} \rangle=2^{2j}2^m  \langle \psi^*_{2;j,k},  \psi_{2;l,n} \rangle=2^{2j+m} \delta_{jl,kn},
 $$
 we have that   $c_{l,n}=0$ for all $l \in \N_0$ and $n \in \zz$.

\section{Sampling characterization of Besov-Triebel-Lizorkin spaces via piecewise cubic Faber splines}

In this Section we prove Theorem about sampling characterization of Besov-Triebel-Lizorkin spaces. We recall that definition of these spaces is given in Appendix A.

The following theorem holds.
\begin{satz}\label{charact3} \begin{itemize}
		\item [(i)]
	Let $0< p,\theta\leq \infty$, $p>1/4$ and $1/p<r<\min\{3+1/p,4\}$. Then every compactly supported $f \in B_{p,\theta}^r$ can be represented by the series (\ref{repr-b}), which is convergent unconditionally in the space $B_{p,\theta}^{r-\varepsilon}$  for every $\varepsilon>0$. If $\max\{p,\theta\}<\infty$ we have unconditional convergence in the space $B_{p,\theta}^{r}$.  Moreover, the following norms are equivalent
	\begin{equation}\label{norm-equi}
	\|f\|_{B_{p,\theta}^r} \asymp \|\lambda(f) \|_{b_{p,\theta}^r} .
	\end{equation}
		\item [(ii)] 	Let $1/4< p,\theta\leq \infty$, $p\neq \infty$, and $\max\{1/p,1/\theta\}<r<3$. Then every compactly supported $f \in F_{p,\theta}^r$ can be represented by the series (\ref{repr-b}), which is convergent unconditionally in the space $F_{p,\theta}^{r-\varepsilon}$  for every $\varepsilon>0$. If $\theta<\infty$ we have unconditional convergence in the space $F_{p,\theta}^{r}$.  Moreover, the following norms are equivalent
		\begin{equation}\label{norm-equi-f}
	\|f\|_{F_{p,\theta}^r} \asymp	\|\lambda(f) \|_{f_{p,\theta}^r}  .
		\end{equation}
		\end{itemize}
\end{satz}

First we prove some auxiliary statements. We apply  some known technique that was also used in \cite{Glen2018}, \cite{HMOT2016} and \cite{TV2019}. Let $\lambda=(\lambda_{j,k})_{j \in \N_{-1}, k \in \zz}$ be some sequence of real numbers that satisfy certain conditions (later we specify that $\lambda \in b_{p,\theta}^r$ or $\lambda \in f_{p,\theta}^r$). We denote
\begin{equation}\label{series}
f:=\sum \limits_{j \in \N_{-1}} \sum \limits_{k \in \zz} \lambda_{j,k} s_{j,k}.
\end{equation}
This formal series converges in  $S'(\re)$ due to assumptions on $\lambda$. The estimates below shows that (\ref{series}) exists in $S'(\re)$.

\begin{prop}\label{charact1}
Let $0< p,\theta\leq \infty$, $\max\{1/p-1,0\}<r<3+1/p$ and a sequence $\lambda \in b_{p,\theta}^r$. Then the series (\ref{series}) converges unconditionally in the space $B_{p,\theta}^{r-\varepsilon}$  for every $\varepsilon>0$. If $\max\{p,\theta\}<\infty$ we have unconditional convergence in the space $B_{p,\theta}^{r}$. Moreover, the following inequality holds
\begin{equation}\label{ch1}
\|f\|_{B_{p,\theta}^r} \lesssim \|\lambda \|_{b_{p,\theta}^r}.
\end{equation}
\end{prop}
\begin{proof}	
First we prove the inequality (\ref{ch1}) for the case $\theta<\infty$. For $\theta=\infty$ the proof is similar.
We denote $f_j:=\sum \limits_{k \in \zz} \lambda_{j,k} s_{j,k}$ for $j \in \N_{-1}$. Then
\begin{equation}\label{repr}
f=\sum \limits_{l \in \zz} f_{j+l}.
\end{equation}
By using characterization of Besov spaces via local means (Theorem \ref{local-mean})
and $u$-triangle inequality with $u:=\min\{p,\theta,1\}$ we have
\begin{align*}
\|f\|_{B_{p,\theta}^r}&\asymp\Big( \sum\limits_{j \in \N_0} 2^{\theta j r} \|\Psi_j * f\|_p^\theta \Big)^{1/\theta}  =\Big( \sum\limits_{j \in \N_0} 2^{\theta j r} \Big\|\Psi_j * \Big(\sum \limits_{l \in \zz} \sum \limits_{k \in \zz} \lambda_{j+l,k} s_{j+l,k} \Big)\Big\|_p^\theta \Big)^{1/\theta} \\
& \leq \Big( \sum \limits_{l \in \zz} \Big( \sum\limits_{j \in \N_0} 2^{\theta j r} \Big\| \sum \limits_{k \in \zz} \lambda_{j+l,k} (\Psi_j*s_{j+l,k}) \Big\|_p^\theta   \Big)^{u/\theta}\Big)^{1/u}.
\end{align*}
By using inequality (\ref{conv-b-1}) we can proceed for $v=\min\{p,1\}$
\begin{align}
\|f\|_{B_{p,\theta}^r}& \lesssim \Bigg(\sum \limits_{l \in \zz} \Bigg( \sum\limits_{j \in \N_0} 2^{\theta j r} \Big\| \sum \limits_{k \in \zz} \lambda_{j+l,k} 2^{-\alpha|l|} \sum \limits_{n \in \zz} |a_n| \chi_{A_{j+l,k+n}}(\cdot) \Big\|_p^\theta   \Bigg)^{u/\theta} \Bigg)^{1/u}\notag \\
& \leq \Bigg( \sum \limits_{l \in \zz} \Bigg( \sum\limits_{j \in \N_0} 2^{\theta j r} 2^{-\alpha|l| \theta} \Big( \sum \limits_{n \in \zz} |a_n|^v \Big\| \sum \limits_{k \in \zz} \lambda_{j+l,k}  \chi_{A_{j+l,k+n}}(\cdot) \Big\|^v_p\Big)^{\theta/v}  \Bigg)^{u/\theta}\Bigg)^{1/u} .  \label{norm}
\end{align}
Further we consider the following norm $\Big\| \sum \limits_{k \in \zz} \lambda_{j+l,k}  \chi_{A_{j+l,k+n}}(\cdot) \Big\|_p$. For $x \in \re$ since $A_{j+l,n+k}\subset \bigcup\limits_{|i-k|\lesssim 2^{l_+}} I_{j+l_+,i+n}$ we can write
$$
\Big|\sum \limits_{k \in \zz} \lambda_{j+l,k}  \chi_{A_{j+l,k+n}}(x) \Big|^p \leq \Big|\sum \limits_{k \in \zz} |\lambda_{j+l,k}|  \chi_{A_{j+l,k+n}}(x) \Big|^p\leq \Big|\sum \limits_{k \in \zz} |\lambda_{j+l,k}| \sum \limits_{i \in G_{l}(k)} \chi_{I_{j+l_+,i+n}}(x) \Big|^p,
$$
where $G_{l}(k):=\{i: \, |i-k|\lesssim 2^{l_+}\}$ with $|G_{l}(k)|\asymp 2^{l+}$. By changing order of summation and on the viewpoint that segments $I_{j+l_+,i+n}$ do not intersect for different $i$ we have
$$
\Big|\sum \limits_{k \in \zz} \lambda_{j+l,k}  \chi_{A_{j+l,k+n}}(x) \Big|^p \leq \Big|\sum \limits_{i \in \zz} \chi_{I_{j+l_+,i+n}}(x) \sum \limits_{k \in G_l(i)} |\lambda_{j+l,k}|  \Big|^p=\sum \limits_{i \in \zz} \chi_{I_{j+l_+,i+n}}(x) \Big( \sum \limits_{k \in G_l(i)} |\lambda_{j+l,k}| \Big)^p.
$$
From the H\"{o}lder inequality for $p>1$ we obtain
$$
\Big|\sum \limits_{k \in \zz} \lambda_{j+l,k}  \chi_{A_{j+l,k+n}}(x) \Big|^p \lesssim   2^{l_+(p-1)} \sum \limits_{i \in \zz} \chi_{I_{j+l_+,i+n}}(x) \sum \limits_{k \in G_l(i)} |\lambda_{j+l,k}|^p.
$$
For $p<1$ we use the embedding $l_p \hookrightarrow l_1$ to get
$$
\Big|\sum \limits_{k \in \zz} \lambda_{j+l,k}  \chi_{A_{j+l,k+n}}(x) \Big|^p \leq  \sum \limits_{i \in \zz} \chi_{I_{j+l_+,i+n}}(x) \sum \limits_{k \in G_l(i)} |\lambda_{j+l,k}|^p.
$$
By using last inequality we can write for the norm
\begin{align*}
\Big\|\sum \limits_{k \in \zz} \lambda_{j+l,k}  \chi_{A_{j+l,k+n}} \Big\|_p^p &= \int\limits_{\re} \Big|\sum \limits_{k \in \zz} \lambda_{j+l,k}  \chi_{A_{j+l,k+n}}(x) \Big|^p dx \\
& \lesssim  2^{l_+(p-1)_+}  \sum \limits_{i \in \zz} \int\limits_{\re} \chi_{I_{j+l_+,i+n}} (x) dx \sum \limits_{k \in G_l(i)} |\lambda_{j+l,k}|^p\\
&= 2^{l_+(p-1)_+}  \sum \limits_{i \in \zz} \int\limits_{2^{-j-l_+}(i+n)}^{2^{-j-l_+}(i+n+1)} 1 \, dx \sum \limits_{k\in G_l(i)} |\lambda_{j+l,k}|^p\\
&=2^{l_+(p-1)_+} 2^{-j-l_+} \sum \limits_{i \in \zz}\sum \limits_{k\in G_l(i)} |\lambda_{j+l,k}|^p\\
& \asymp 2^{l_+(p-1)_+}2^{-j} \sum \limits_{k \in \zz}|\lambda_{j+l,k}|^p.
\end{align*}

 By using this inequality we can continue estimation of (\ref{norm})
\begin{align*}
\|f\|_{B_{p,\theta}^r}& \lesssim   \Bigg(  \sum \limits_{l \in \zz} \Bigg( \sum\limits_{j \in \N_0} 2^{\theta j r} 2^{-\alpha|l| \theta} \Big( \sum \limits_{n \in \zz} |a_n|^v 2^{vl_+(1-1/p)_+}2^{-jv/p} \Big(\sum \limits_{k \in \zz}|\lambda_{j+l,k}|^p \Big)^{v/p} \Big)^{\theta/v}  \Bigg)^{u/\theta} \Bigg)^{1/u} \\
& = \Big(\sum \limits_{n \in \zz} |a_n|^v \Big)^{1/v} \cdot  \Bigg(  \sum \limits_{l \in \zz} \Bigg( \sum\limits_{j \in \N_0} 2^{\theta j r} 2^{-\alpha|l| \theta} 2^{l_+(1-1/p)_+\theta}2^{-j\theta/p} \Big(\sum \limits_{k \in \zz}|\lambda_{j+l,k}|^p \Big)^{\theta/p} \Bigg)^{u/\theta}\Bigg)^{1/u}.
\end{align*}
From definition of coefficients $a_n$ we conclude that $\sum \limits_{n \in \zz} |a_n|^v < \infty$, so we can proceed as follows
 \begin{align*}
 \|f\|_{B_{p,\theta}^r}& \lesssim   \Bigg( \sum \limits_{l \in \zz} 2^{-\alpha|l|u} 2^{l_+(1-1/p)_+u} \Bigg( \sum\limits_{j \in \N_0} 2^{\theta j (r-1/p)}   \Big(\sum \limits_{k \in \zz}|\lambda_{j+l,k}|^p \Big)^{\theta/p} \Bigg)^{u/\theta}\Bigg)^{1/u}\\
 &=\Bigg(\sum \limits_{l \in \zz} 2^{-\alpha|l|} 2^{l_+(1-1/p)_+} 2^{-l(r-1/p)} \Big( \sum\limits_{j \in \N_0} 2^{\theta (j+l) (r-1/p)}   \Big(\sum \limits_{k \in \zz}|\lambda_{j+l,k}|^p \Big)^{\theta/p} \Big)^{u/\theta}\Bigg)^{1/u}\\
 & \leq \Big( \sum \limits_{l \in \zz} 2^{-\alpha|l|u} 2^{l_+(1-1/p)_+u} 2^{-l(r-1/p)u} \Big)^{1/u} \|\lambda\|_{b_{p,\theta}^r}.
 \end{align*}
Due to the choice of the parameter $\max\{1/p-1,0\}<r<3+1/p$ the series
$$
 \sum \limits_{l \in \zz} 2^{-\alpha|l|u} 2^{l_+(1-1/p)_+u} 2^{-l(r-1/p)u}
$$
 is convergent. Therefore, inequality (\ref{ch1}) holds.

Let us further prove the unconditional convergence of the series (\ref{series}) in the space $B_{p,\theta}^r$ when $\max\{p,\theta\}<\infty$.
 By $\nabla$ we define a set of indices for the basis  $s_{j+l,k}$, i.e
 $$
 \nabla=\{(j,k): j\in \N_{-1}, k \in \zz\}.
 $$
 We consider the set of sequences
 $$
 \Theta=\{\mathcal{A}=(\mathcal{A}_n)_{n \in \N}: \,  \mathcal{A}_n \subset \nabla,|\mathcal{A}_n|=n,  \mathcal{A}_n\subset \mathcal{A}_{n+1}, \bigcup\limits_{n=1}^\infty\mathcal{A}_n=\nabla \}.
 $$
 Each $\mathcal{A} \in  \Theta$ defines some order of summation of the series (\ref{series}).
 By $S_n$ we denote the following partial sum
 $$
 S_n:=\sum \limits_{(j,k) \in \mathcal{A}_n}  \lambda_{j,k} s_{j,k}.
 $$
 According to first part of the proof we have that
  $$
 \|f-S_n\|_{B_{p,\theta}^r} \lesssim \|\lambda_{j,k}|_{(j,k)\in \nabla  \setminus \mathcal{A}_n}\|_{b_{p,\theta}^r}.
 $$
 Since all finite sequences are dense in the space $b_{p,\theta}^r$, $\max\{p,\theta\}<\infty$ we have that if $n$ is large enough
 $$
 \|\lambda_{j,k}|_{(j,k)\in \nabla  \setminus \mathcal{A}_n}\|_{b_{p,\theta}^r} < \varepsilon,
 $$
 what together with arbitrary choice of $\mathcal{A}$ finishes the proof.

Let now $0<p,\theta\leq\infty$ and $\lambda \in b_{p,\theta}^r$. Then by using H\"{o}lder's inequality with respect to index $j$ in the definition of the norm of $b_{p,\theta}^{r-\varepsilon}$ it is easy to show that
$$
\lim \limits_{n\rightarrow \infty}  \|\lambda_{j,k}|_{(j,k)\in \nabla  \setminus \mathcal{A}_n}\|_{b_{p,\theta}^{r-\varepsilon}} =0.
$$
By using inequality (\ref{ch1}) that is already proven we write
$$
\|f-S_n\|_{B_{p,\theta}^{r-\varepsilon}} \lesssim   \|\lambda_{j,k}|_{(j,k)\in \nabla  \setminus\mathcal{A}_n}\|_{b_{p,\theta}^{r-\varepsilon}} \rightarrow 0
$$
which finishes the proof.
\end{proof}

Now we prove the analogue of this proposition for $F$-spaces.

\begin{prop}\label{charactf1}
	Let $0< p,\theta\leq \infty$, $p\neq \infty$,  $\max\{1/\theta-1,1/p-1,0\}<r<3$ and a sequence $\lambda \in f_{p,\theta}^r$. Then the series (\ref{series})
	converges unconditionally in the space $F_{p,\theta}^{r-\varepsilon}$  for every $\varepsilon>0$. If $\theta<\infty$ we have unconditional convergence in the space $F_{p,\theta}^{r}$. Moreover, the following inequality holds
	\begin{equation}\label{ch1-f}
	\|f\|_{F_{p,\theta}^r} \lesssim \|\lambda \|_{f_{p,\theta}^r}.
	\end{equation}
\end{prop}
\begin{proof}
We use $u$-triangle inequality with $u=\min\{1,p,q\}$, representation (\ref{repr}) and  Theorem \ref{local-mean}
\begin{align*}
\|f\|_{F_{p,\theta}^{r}} &\asymp
\Big\|\Big(\sum\limits_{j \in \N_0} 2^{\theta r j} |\Psi_{j}*f|^\theta   \Big)^{1/\theta}\Big\|_p \\
&= \Big\|\Big(\sum\limits_{j \in \N_0} 2^{\theta r j} \Big|\Psi_{j}*\Big(\sum \limits_{l \in \zz} \sum \limits_{k \in \zz} \lambda_{j+l,k} s_{j+l,k} \Big)\Big|^\theta   \Big)^{1/\theta}\Big\|_p \\
& \leq \Bigg( \sum \limits_{l \in \zz} \Big\|\Big(\sum\limits_{j \in \N_0} 2^{\theta r j} \Big|\Psi_{j}*\Big( \sum \limits_{k \in \zz} \lambda_{j+l,k} s_{j+l,k} \Big)\Big|^\theta   \Big)^{1/\theta}\Big\|_p^u\Bigg)^{1/u}\\
& \leq \Bigg( \sum \limits_{l \in \zz} \Big\|\Big(\sum\limits_{j \in \N_0} 2^{\theta r j} \Big( \sum \limits_{k \in \zz} |\lambda_{j+l,k}| |\Psi_{j}*s_{j+l,k}| \Big)^\theta   \Big)^{1/\theta}\Big\|_p^u\Bigg)^{1/u}.
\end{align*}
By using inequality (\ref{conv-b-2}) we obtain
$$
\|f\|_{F_{p,\theta}^{r}} \lesssim \Bigg(  \sum \limits_{l \in \zz} 2^{-\alpha |l|u} \Big\|\Big(\sum\limits_{j \in \N_0} 2^{\theta r j} \Big( \sum \limits_{k \in \zz} |\lambda_{j+l,k}| (1+2^{\min\{j,j+l\}}|x-x_{j+l,k}|)^{-R} \Big)^\theta   \Big)^{1/\theta}\Big\|_p^u \Bigg)^{1/u}.
$$
From the following property (\cite[Lem. 7.1]{Kyr2003})
$$
\sum\limits_{k \in \zz} |\lambda_{j+l,k}| (1+2^{\min\{j,j+l\}}|x-x_{j+l,k}|)^{-R}\lesssim 2^{l_+/\tau} \Big[M\Big|\sum\limits_{k \in \zz} \lambda_{j+l,k} \chi_{j+l,k}\Big|^\tau \Big]^{1/\tau}(x)
$$
for $0<\tau \leq 1$ and $R>1/\tau$, we get
$$
\|f\|_{F_{p,\theta}^{r}} \lesssim \Bigg(  \sum \limits_{l \in \zz} 2^{-\alpha |l|u} 2^{u l_+/\tau} \Big\|\Big(\sum\limits_{j \in \N_0} 2^{\theta r j}  \Big[M\Big|\sum\limits_{k \in \zz} \lambda_{j+l,k} \chi_{j+l,k}\Big|^\tau \Big]^{\theta/\tau} \Big)^{1/\theta}\Big\|_p^u \Bigg)^{1/u}.
$$
It is obvious that $\Big\|\Big( \sum \limits_l \big[M|f_l|^\tau\big]^{\theta/\tau}\Big)^{1/\theta} \Big\|_p=\Big\|\Big( \sum \limits_l \big[M|f_l|^\tau\big]^{\theta/\tau}\Big)^{\tau/\theta} \Big\|_{p/\tau}^{1/\tau}$.
We assume that $\min\{\theta/\tau,p/\tau\}>1$.
By using the Hardy-Littlewood maximal inequality we have
\begin{align*}
\|f\|_{F_{p,\theta}^{r}} &\lesssim \Bigg( \sum \limits_{l \in \zz} 2^{-\alpha |l|u} 2^{u l_+/\tau} \Big\|\Big(\sum\limits_{j \in \N_0} 2^{\theta r j}  \Big|\sum\limits_{k \in \zz} \lambda_{j+l,k} \chi_{j+l,k}\Big|^{\theta} \Big)^{1/\theta}\Big\|_p^u \Bigg)^{1/u}\\
&=\Bigg( \sum \limits_{l \in \zz} 2^{-\alpha |l|u} 2^{u l_+/\tau}2^{-rlu} \Big\|\Big(\sum\limits_{j \in \N_0} 2^{\theta r (j+l)}  \Big|\sum\limits_{k \in \zz} \lambda_{j+l,k} \chi_{j+l,k}\Big|^{\theta} \Big)^{1/\theta}\Big\|_p^u\Bigg)^{1/u}\\
& \leq \Big(\sum \limits_{l \in \zz} 2^{-\alpha |l|u} 2^{ ul_+/\tau}2^{-rlu}\Big)^{1/u} \|\lambda\|_{f_{p,\theta}^r}
\end{align*}
for $\tau<\min\{1,p,\theta\}$.
If $\max\{1/\theta-1,1/p-1,0\}\leq 1/\tau-1<r<3$ then the series $\sum \limits_{l \in \zz} 2^{-\alpha |l|u} 2^{ ul_+/\tau}2^{-rlu}$ converges.

Now we prove unconditional convergence. We start with the case when $\theta <\infty$. We use notations from Proposition \ref{charact1}.  We know that
$$
\|f-S_n\|_{F_{p,\theta}^r} \lesssim \|\lambda_{j,k}|_{(j,k)\in \nabla  \setminus \mathcal{A}_n}\|_{f_{p,\theta}^r}.
$$
From density of finite sequences in the space $l_\theta$ we have that
$$
\Big(\sum_{(j,k)\in \nabla  \setminus \mathcal{A}_n } 2^{\theta r j} |\lambda_{j,k}|^\theta \chi_{j,k}\Big)^{1/\theta}\rightarrow 0, \text{  if  }  n\rightarrow\infty.
$$
Since for all $n \in \N$
$$
\Big(\sum_{(j,k)\in \nabla  \setminus \mathcal{A}_n } 2^{\theta r j} |\lambda_{j,k}|^\theta \chi_{j,k}\Big)^{1/\theta} \leq \Big(\sum_{(j,k)\in \nabla } 2^{\theta r j} |\lambda_{j,k}|^\theta \chi_{j,k}\Big)^{1/\theta} \in L_p,
$$
then according to Lebesgue dominated convergence theorem we may write that
$$
\|\lambda_{j,k}|_{(j,k)\in \nabla  \setminus \mathcal{A}_n}\|_{f_{p,\theta}^r}\rightarrow 0, \text{  if  }  n\rightarrow\infty.
$$
When $\theta=\infty$ we use similar consideration as in Proposition \ref{charact1}.
\end{proof}

\begin{prop}\label{charact2} \begin{itemize}
		\item [(i)] 	Let $0< p,\theta\leq \infty$, $p>1/4$, $1/p<r<4$ and  $f \in B_{p,\theta}^r$. Then the inequality
		\begin{equation}\label{ch2}
		\|\lambda(f) \|_{b_{p,\theta}^r}\lesssim \|f\|_{B_{p,\theta}^r}
		\end{equation}
		holds.
		\item [(ii)] 	Let $1/4< p,\theta\leq \infty$, $p\neq \infty$, $\max\{1/p,1/\theta\}<r<4$ and  $f \in F_{p,\theta}^r$. Then the inequality
		\begin{equation}\label{chf2}
		\|\lambda(f) \|_{f_{p,\theta}^r}\lesssim \|f\|_{F_{p,\theta}^r}
		\end{equation}
		holds.
	\end{itemize}
\end{prop}
\begin{proof}
We use representation (\ref{frep}) in the following form
\begin{equation}\label{delta-char}
f=\sum \limits_{l \in \zz} \delta_{j+l}[f], \ \ j \in \N_0.
\end{equation}
We give a proof for the case $\theta<\infty$. For $\theta=\infty$ one can obtain the results by using similar technique with trivial modification.

First we prove one additional inequality. We denote $F_{j,l}(x):=\sum\limits_{k \in \zz} \lambda_{j,k}\left( \delta_{j+l}[f] \right) \chi_{j,k}(x)$, $x \in \re$.
For $x \in I_{j,k}$ we  have that
$$
|F_{j,l}(x)| \leq |\lambda_{j,k}\left( \delta_{j+l}[f] \right)|.
$$
Let first $j\geq 0$. Then
$$
|F_{j,l}(x)|\leq \frac{1}{6} \left(\left|\Delta^4_{2^{-j-1}}\delta_{j+l}[f]\left( \dfrac{2k}{2^{j+1}}\right)\right|+4\left|\Delta^4_{2^{-j-1}}\delta_{j+l}[f]\left( \dfrac{2k+1}{2^{j+1}}\right)\right|+\left|\Delta^4_{2^{-j-1}}\delta_{j+l}[f]\left( \dfrac{2k+2}{2^{j+1}}\right)\right| \right).
$$

By using Lemma \ref{peetre-ineq} we get for some bandlimited function $g$ with $\mathcal{F}g\subset [-A2^{j+l},B2^{j+l}]$
$$
|\Delta^4_{2^{-j-1}} g(x_{j,k})|\lesssim \min\{1,2^{4l}\} \max\{1,2^{al}\} P_{2^{l+j},a}g (x_{j,k}).
$$
From this inequality for $l<0$ and $|x-x_{j,k}|\leq 2^{-j}$ we get
\begin{align}
|\Delta^4_{2^{-j-1}} g(x_{j,k})|& \lesssim 2^{4l} P_{2^{l+j},a}g (x_{j,k})  \leq 2^{4l} \sup\limits_{y \in \re} \frac{|g(y)|}{(1+2^{l+j}|y-x|)^a} (1+2^{l+j}|x-x_{j,k}|)^a\notag \\
& \lesssim 2^{4l} \sup\limits_{y \in \re} \frac{|g(y)|}{(1+2^{l+j}|y-x|)^a}=2^{4l} P_{2^{l+j},a}g (x). \label{dif1}
\end{align}

For $l\geq 0$ and $|x-x_{j,k}|\leq 2^{-j}$ by using definition of the 4th order  difference we write
\begin{align*}
|\Delta^4_{2^{-j-1}} g(x_{j,k})|& \lesssim \sup\limits_{|y|\lesssim 2^{-j}} |g(x_{j,k}+y)| \lesssim \sup\limits_{|y|\lesssim 2^{-j}} \frac{|g(x_{j,k}+y)|}{(1+2^j|y|)^a} \leq P_{2^j,a}(x_{j,k}) \\
&= \sup\limits_{y \in \re} \frac{|g(y)|}{(1+2^j|y-x|)^a}\frac{(1+2^j|y-x|)^a}{(1+2^j|y-x_{j,k}|)^a}  \\
&\leq \sup\limits_{y \in \re} \frac{|g(y)|}{(1+2^j|y-x|)^a} \frac{(1+2^j|y-x_{j,k}|)^a(1+2^j|x-x_{j,k}|)^a}{(1+2^j|y-x_{j,k}|)^a} \lesssim P_{2^j,a}(x) .
\end{align*}
Since for $l\geq 0$ we have
\begin{equation}\label{pin}
P_{2^j,a}(x)=\sup\limits_{y \in \re} \frac{|g(y)|}{(1+2^{j+l}|y-x|)^a} \frac{(1+2^{j+l}|y-x|)^a}{(1+2^{j}|y-x|)^a}  \leq 2^{la}P_{2^{l+j},a}g (x),
\end{equation}
then
\begin{equation}\label{dif2}
|\Delta^4_{2^{-j-1}} g(x_{j,k})| \lesssim 2^{la}P_{2^{l+j},a}g (x).
\end{equation}

If $x \in I_{j,k}$ then not only $|x-x_{j,k}|\leq2^{-j}$, but also $|x-x_{j,k+1}|\leq2^{-j}$ and $|x-x_{j+1,2k+1}|\leq 2^{-j}$. Therefore inequalities (\ref{dif1}) and (\ref{dif2}) hold  also for differences $\Delta^4_{2^{-j-1}}g\left( x_{j+1,2k+1}\right)$ and $\Delta^4_{2^{-j-1}}g\left( x_{j,k+1}\right)$.

Let now $j=-1$. In this case $x \in [k-1/2,k+1/2]$ and
$$
|F_{j,l}(x)| \leq |\lambda_{j,k}\left( \delta_{j+l}[f] \right)| \leq | \delta_{j+l}[f](k)|.
$$
Let again $g$ be some bandlimited function. We consider only the case $l\geq 0$ because $g\equiv 0$ for $l<0$ and there is nothing to prove. So, for $l\geq 0$
\begin{align*}
|g(k)| \leq \sup \limits_{|y|\leq 1} |g(x+y)| =	\sup \limits_{|y|\leq 1} \dfrac{|g(x+y)|}{(1+2^j|y|)^a} (1+2^j|y|)^a \lesssim P_{2^j,a}(x).
\end{align*}	
Using (\ref{pin}) we get
 \begin{equation}\label{dif3}
 |g(k)| \lesssim 2^{la}P_{2^{l+j},a}g (x).
 \end{equation}

Since segments $I_{j,k}$ do not intersect for fixed $j$ and different $k$ from (\ref{dif1}), (\ref{dif2}) and (\ref{dif3})  we conclude that for $x \in \re$
\begin{equation}\label{f-ineq}
|F_{j,l}(x)|\lesssim \min\{2^{4l},1\} \max\{2^{al},1\} P_{2^{j+l},a}\delta_{j+l}[f](x).
\end{equation}

 Let us now prove part (i). From definition of the space of sequences $b_{p,\theta}^r$ by using the $u$-triangle inequality with $u=\min\{p,\theta,1\}$ we can write
\begin{align}
\|\lambda(f)\|_{b_{p,\theta}^r}& = 	\Bigg(\sum\limits_{j \in \N_{-1}} 2^{\theta r j} \Big\|\sum\limits_{k \in \zz} \lambda_{j,k}(f) \chi_{j,k} \Big\|_p^\theta \Bigg)^{1/\theta} \notag\\
&=	\Bigg(\sum\limits_{j \in \N_{-1}} 2^{\theta r j} \Big\|\sum\limits_{k \in \zz} \lambda_{j,k}\Big(\sum \limits_{l \in \zz} \delta_{j+l}[f] \Big) \chi_{j,k} \Big\|_p^\theta \Bigg)^{1/\theta} \notag\\
& \leq \Bigg( \sum\limits_{l \in \zz} \Big( \sum\limits_{j\in \N_{-1}} 2^{\theta r j} \Big\|\sum\limits_{k \in \zz} \lambda_{j,k}\left( \delta_{j+l}[f] \right) \chi_{j,k} \Big\|_p^\theta \Big)^{u/\theta}\Bigg)^{1/u}. \label{lamb-norm}
\end{align}

By using Lemma \ref{peetre-ineq1} and inequality (\ref{f-ineq}) we have that for $a>1/p$
$$
\|F_{j,l}\|_p  \lesssim \min\{2^{4l},1\} \max\{2^{al},1\} \, \|P_{2^{j+l},a}\delta_{j+l}[f]\|_p \lesssim \min\{2^{4l},1\} \max\{2^{al},1\} \, \|\delta_{j+l}[f]\|_p.
$$

 We denote $\beta=a$ if $l\geq 0$ and $\beta=4$ if $l<0$. Now we can proceed estimation (\ref{lamb-norm})
\begin{align*}
\|\lambda(f)\|_{b_{p,\theta}^r}
&\leq \Bigg(\sum\limits_{l \in \zz} \Big( \sum\limits_{j\in \N_{-1}} 2^{\theta r j} \|F_{j,l}\|_p^\theta \Big)^{u/\theta}\Bigg)^{1/u}\\
&\lesssim \Bigg(\sum\limits_{l \in \zz} \Big( \sum\limits_{j\in \N_{-1}} 2^{\theta r j} 2^{\beta \theta l} \|\delta_{j+l}[f]\|_p^\theta \Big)^{u/\theta}\Bigg)^{1/u}\\
& = \Bigg(\sum\limits_{l \in \zz} 2^{\beta l u} 2^{-rl u} \Big( \sum\limits_{j\in \N_{-1}} 2^{\theta r (j+l)}  \|\delta_{j+l}[f]\|_p^\theta \Big)^{u/\theta}\Bigg)^{1/u}\\
&\leq\Big(\sum\limits_{l \in \zz} 2^{(\beta-r) l u} \Big)^{1/u}\|f\|_{B_{p,\theta}^r}.
\end{align*}
Now if $r$ satisfies $1/p<a<r<4$ with $p>1/4$ we have that the series $\sum\limits_{l \in \zz} 2^{(\beta-r) l u}$ converges and inequality (\ref{ch2}) holds.

Now we prove part (ii). We use representation (\ref{delta-char}) and the $u$-triangle inequality
\begin{align*}
\|\lambda(f)\|_{f_{p,\theta}^{r}}&=
\Big\| \Big(\sum\limits_{j \in \N_{-1}} 2^{\theta r j} \Big|\sum\limits_{k \in \zz} \lambda_{j,k}(f) \chi_{j,k} \Big|^\theta \Big)^{1/\theta}\Big\|_p\\
&=\Big\| \Big(\sum\limits_{j \in \N_{-1}} 2^{\theta r j} \Big|\sum\limits_{k \in \zz} \lambda_{j,k}\Big(\sum \limits_{l \in \zz} \delta_{j+l}[f]\Big) \chi_{j,k} \Big|^\theta \Big)^{1/\theta}\Big\|_p\\
& \leq \Bigg(\sum \limits_{l \in \zz} \Big\| \Big(\sum\limits_{j \in \N_{-1}} 2^{\theta r j} \Big|\sum\limits_{k \in \zz} \lambda_{j,k}\Big( \delta_{j+l}[f]\Big) \chi_{j,k} \Big|^\theta \Big)^{1/\theta}\Big\|_p^u \Bigg)^{1/u}.
\end{align*}
By using inequality (\ref{f-ineq}) and Lemma \ref{peetre-ineq2} we write for $a>\max\{1/p,1/\theta\}$
\begin{align*}
\|\lambda(f)\|_{f_{p,\theta}^{r}}& \lesssim \Bigg(\sum \limits_{l \in \zz} 2^{\beta l u} \Big\| \Big(\sum\limits_{j \in \N_{-1}} 2^{\theta r j} \big(P_{2^{j+l},a}\delta_{j+l}[f]\big)^\theta \Big)^{1/\theta}\Big\|_p^u\Bigg)^{1/u} \\
& \lesssim \Bigg(\sum \limits_{l \in \zz} 2^{\beta l u} \Big\| \Big(\sum\limits_{j \in \N_{-1}} 2^{\theta r j} \big(\delta_{j+l}[f]\big)^\theta \Big)^{1/\theta}\Big\|_p^u \Bigg)^{1/u}\\
& =\Bigg(\sum \limits_{l \in \zz} 2^{\beta l u}  2^{-lru}\Big\| \Big(\sum\limits_{j \in \N_{-1}} 2^{\theta r (j+l)} \big(\delta_{j+l}[f]\big)^\theta \Big)^{1/\theta}\Big\|_p^u\Bigg)^{1/u}\\
&\leq \Big(\sum \limits_{l \in \zz} 2^{\beta l u}  2^{-lru} \Big)^{1/u}\|f\|_{F_{p,\theta}^r},
\end{align*}
Due to the choice of the parameter $\max\{1/p,1/\theta\}<a<r<4$ with $p,\theta>1/4$ the series $\sum \limits_{l \in \zz} 2^{\beta lu}  2^{-lru}$ is convergent and inequality (\ref{chf2}) holds.
\end{proof}

\textbf{Proof of Theorem \ref{charact3}.} We give a short proof of part (i). Proof of part (ii) may be obtained in similar way by simple replacement of Proposition \ref{charact1} by Proposition \ref{charactf1} and  part (i) of Proposition \ref{charact2} by part \ref{charact2} (ii).

Relation (\ref{norm-equi})	follows from Propositions \ref{charact1} and \ref{charact2} (i).
We prove convergence for the case $\max\{p,\theta\}<\infty$. Since $f \in B_{p,\theta}^r$, according to Proposition \ref{charact2} (i) we have that $\lambda(f) \in b_{p,\theta}^r$. Then Proposition \ref{charact1} implies that the series $\sum \limits_{j \in \N_{-1}} \sum \limits_{k \in \zz} \lambda_{j,k}(f) s_{j,k}$ converges unconditionally in $B_{p,\theta}^r$ to some function $g$. But since $B_{p,\theta}^r \subset C_0(\re)$ because of the choice of the parameter $r>1/p$ then according to Theorem  \ref{conv-c}  we have uniform convergence of the series $\sum \limits_{j \in \N_{-1}} \sum \limits_{k \in \zz} \lambda_{j,k}(f) s_{j,k}$ to $f$. That implies that $f\equiv g$.
$\blacksquare$

\section{Characterization of Besov-Triebel-Lizorkin spaces  via Chui-Wang wavelets}
Further we use the following notations. We denote $\psi_{j,k}:=\psi_{2;j,k}$ for $j \in N_0$ and $\psi_{-1,k}:=N_2(\cdot+1-k)$. Analogically,  $\psi^*_{j,k}:=\psi^*_{2;j,k}$ for $j \in N_0$ and $\psi^*_{-1,k}:=N_2^*(\cdot-k)$. Let $\mu_{j,k}(f):=\langle f, 2^j\psi_{j,k}\rangle$. Then for each $f \in L_2(\re)$ the following expansion holds
\begin{equation}\label{repr-psi}
f=\sum_{j\in \N_{-1}} \sum\limits_{k \in \zz} \mu_{j,k}(f) \psi^*_{j,k},
\end{equation}
where convergence is understood in term of the $L_2$ norm.

The main goal of this section is to prove the following theorem.
\begin{satz}\label{charact3-psi} \begin{itemize}
		\item [(i)]
		Let $0< p,\theta\leq \infty$, $p>1/4$ and $1/p-2<r<\min\{1+1/p,2\}$. Then $f \in B_{p,\theta}^r$ can be represented by the series (\ref{repr-psi}), which convergent unconditionally in the space $B_{p,\theta}^{r-\varepsilon}$. If $\max\{p,\theta\}<\infty$ we have unconditional convergence in the space $B_{p,\theta}^{r}$.  Moreover, the following norms are equivalent
		\begin{equation}\label{norm-equi-psi}
	\|f\|_{B_{p,\theta}^r} \asymp 	\|\mu(f) \|_{b_{p,\theta}^r} .
		\end{equation}
		\item [(ii)] 	Let $1/4< p,\theta\leq \infty$, $p\neq \infty$, and $\max\{1/p-2,1/\theta-2\}<r<1$. Then $f \in F_{p,\theta}^r$ can be represented by the series (\ref{repr-psi}), which convergent unconditionally in the space $F_{p,\theta}^{r-\varepsilon}$. If $\theta<\infty$ we have unconditional convergence in the space $F_{p,\theta}^{r}$.  Moreover, the following norms are equivalent
		\begin{equation}\label{norm-equi-f-psi}
		\|f\|_{F_{p,\theta}^r} \asymp \|\mu(f) \|_{f_{p,\theta}^r}  .
		\end{equation}
	\end{itemize}
\end{satz}

\begin{rem}\label{dual-pair}
	Let $\widetilde{p}:=\max\{p,1\}$. Then for the scalar product
	$$
	\langle f, \psi_2\rangle:=\sum\limits_{j \in \N_0}\langle \Psi_j * f, \Lambda_j*\psi_2 \rangle
	$$
	we have
	\begin{align*}
	|\langle f, \psi_2\rangle| &\leq \sum\limits_{j \in \N_0} \|\Psi_j * f\|_{\widetilde{p}} \|\Lambda_j*\psi_2\|_{\widetilde{p}'} \\
	&\leq \sup \limits_{j \in \N_0} 2^{j(r-(1/p-1)_+)} \|\Psi_j * f\|_{\widetilde{p}} \sum\limits_{j \in \N_0} 2^{j(-r+(1/p-1)_+)} \|\Lambda_j*\psi_2\|_{\widetilde{p}'} \\
	& \leq \|f\|_{B_{\widetilde{p},\infty}^{r-(1/p-1)_+}} \|\psi_2\|_{B_{\widetilde{p}',1}^{-r+(1/p-1)_+}}.
	\end{align*}
	By using Theorem \ref{local-mean} it is easy to show that $\psi_2 \in B_{p,\theta}^r$ if $r<1/p+1$. Therefore, due to the choice of parameter $r$ in Theorem \ref{charact3-psi} we have that $\psi_2\in B_{\widetilde{p}',1}^{-r+(1/p-1)_+}$. If $f \in F_{p,\theta}^r$ we use the embedding  $F_{p,\theta}^r\subset B_{\max\{1,p\},\infty}^{r-(1/p-1)_+}$ to conclude that $\|f\|_{B_{\widetilde{p},\infty}^{r-(1/p-1)_+}}<\infty$. Note, that we can choose $\Psi_j$ and $\Lambda_j$ such that $\Psi_j$ is compactly supported on the Fourier side and $\Lambda_j$ is compactly supported on time domain (see \cite{UU2015}).
\end{rem}

First we prove some auxiliaries statements. Let $\mu=(\mu_{j,k})_{j \in \N_{-1}, k \in \zz}$ be some sequence of real numbers that satisfy certain conditions ($\mu \in b_{p,\theta}^r$ or $\mu \in f_{p,\theta}^r$). We denote
	\begin{equation}\label{series-psi}
f:=\sum_{j\in \N_{-1}} \sum\limits_{k \in \zz} \mu_{j,k} \psi^*_{j,k}.
	\end{equation}
Due to assumptions on $\mu$  we have convergence of the  series (\ref{series-psi}) at least at $S'(\re)$.
	
\begin{prop}\label{charact-psi-1}
	Let $0< p,\theta\leq \infty$, $\max\{1,1/p\}-3<r<1+1/p$ and a sequence $\mu \in b_{p,\theta}^r$. Then a series (\ref{series-psi})
	converges unconditionally in the space $B_{p,\theta}^{r-\varepsilon}$. If $\max\{p,\theta\}<\infty$ we have unconditional convergence in the space $B_{p,\theta}^{r}$. Moreover, the following inequality holds
	\begin{equation}\label{ch1-psi}
	\|f\|_{B_{p,\theta}^r} \lesssim \|\mu\|_{b_{p,\theta}^r}.
	\end{equation}
\end{prop}
\begin{proof}
For the proof we use the same technique as in Proposition \ref{charact1} with 	Lemma \ref{convol-chw} instead of Lemma \ref{convol}. Different values of $\alpha$ for positive and negative $l$ lead to different range of smoothness parameter $r$.
\end{proof}		

\begin{prop}\label{charact-psi-1-f}
	Let $0< p,\theta\leq \infty$, $p\neq \infty$, $\max\{1,1/p,1/\theta\}-3<r<1$ and a sequence $\mu \in f_{p,\theta}^r$. Then a series (\ref{series-psi})
	converges unconditionally in the space $F_{p,\theta}^{r-\varepsilon}$. If $\theta<\infty$ we have unconditional convergence in the space $F_{p,\theta}^{r}$. Moreover, the following inequality holds
	\begin{equation}\label{ch1-psi-f}
	\|f\|_{F_{p,\theta}^r} \lesssim \|\mu\|_{f_{p,\theta}^r}.
	\end{equation}
\end{prop}
\begin{proof}
	For the proof we use the same technique as in Proposition \ref{charactf1} with 	Lemma \ref{convol-chw} instead of Lemma \ref{convol}. Different values of $\alpha$ for positive and negative $l$ lead to different range of smoothness parameter $r$.
\end{proof}

\begin{prop}\label{charact-psi-2} \begin{itemize}
		\item [(i)] 	Let $0< p,\theta\leq \infty$, $p>1/4$, $1/p-2<r<2$ and  $f \in B_{p,\theta}^r$. Then the inequality
		\begin{equation}\label{ch-psi-2}
		\|\langle f, 2^j\psi_{j,k}\rangle \|_{b_{p,\theta}^r}\lesssim \|f\|_{B_{p,\theta}^r}
		\end{equation}
		holds.
		\item [(ii)] 	Let $1/4< p,\theta\leq \infty$, $p\neq \infty$, $\max\{1/p-2,1/\theta-2\}<r<2$ and  $f \in F_{p,\theta}^r$. Then the inequality
		\begin{equation}\label{ch-psi-2-f}
		\|\langle f, 2^j\psi_{j,k}\rangle \|_{f_{p,\theta}^r}\lesssim \|f\|_{F_{p,\theta}^r}
		\end{equation}
		holds.	
	\end{itemize}

\end{prop}
\begin{proof}
First we estimate one coefficient $|\langle f,2^j \psi_{j,k}\rangle|$. By properly chosen Littlewood-Paley building blocks we may write
\begin{align}
|\langle f,2^j \psi_{j,k}\rangle| &\leq
\sum\limits_{l \in \zz} \Big|
2^j \langle \Psi_{j+l}\ast f, \Lambda_{j+l}\ast \psi_{j,k} \rangle\Big| \notag\\
&= \sum\limits_{l \in\zz}
\Big|\int_{-\infty}^{\infty}
2^j(\Psi_{j+l}\ast f)(y)\cdot (\Lambda_{j+l}\ast \psi_{j,k})(y)dy\Big|.\label{f5}
\end{align}	
We estimate the inner integral. For $l\geq0$ we have that
$$
|(\Lambda_{j+l}\ast \psi_{j,k})(y)| \lesssim 2^{-l} \chi_{A_{j+l,k}}(y).
$$	
The set $A_{j+l,k}$ here is the union of at most 7 intervals centered at nodes of $\psi_{j,k}$ with lengths $2^{-j-l}$ and factor $2^{-l}$ is due to the smoothness of $\psi_{j,k}$ ($\psi_{j,k} \in B_{\infty,\infty}^1$).
By using this we have
\begin{align*}
\Big|\int_{-\infty}^{\infty}
2^j(\Psi_{j+l}\ast f)(y)\cdot (\Lambda_{j+l}\ast \psi_{j,k})(y)dy\Big|& \lesssim  2^j 2^{-l} \int_{-\infty}^{\infty} |(\Psi_{j+l}\ast f)(y)| \chi_{A_{j+l,k}}(y) dy\\
&=2^j 2^{-l} \int\limits_{A_{j+l,k}} |(\Psi_{j+l}\ast f)(y)| dy\\
&\lesssim 2^j 2^{-l} 2^{-j-l} \sup\limits_{y \in A_{j+l,k}} |(\Psi_{j+l}\ast f)(y)| \\
&= 2^{-2l}  \sup\limits_{y \in A_{j+l,k}} \frac{|(\Psi_{j+l}\ast f)(y)|(1+2^{j+l}|x-y|)^a}{(1+2^{j+l}|x-y|)^a}.
\end{align*}
For $x\in I_{j,k}$ and $y \in A_{j+l,k}$ we have that $|x-y|<2^{-j}$. Therefore
\begin{equation}\label{lpos}
\Big|\int_{-\infty}^{\infty}
2^j(\Psi_{j+l}\ast f)(y)\cdot (\Lambda_{j+l}\ast \psi_{j,k})(y)dy\Big|\lesssim 2^{l(a-2)} P_{2^{j+l},a}(\Psi_{j+l}\ast f)(x).
\end{equation}
For $l<0$
$$
|(\Lambda_{j+l}\ast \psi_{j,k})(y)| \lesssim 2^{3l} \chi_{A_{j+l,k}}(y),
$$
where $A_{j+l,k}$ is a segment centered at $x_{j,k}$ with length $\sim 2^{-j-l}$. Factor $2^{3l}$ is due to vanishing moments of $\psi_{j,k}$.
Therefore, as above
\begin{align}
\Big|\int_{-\infty}^{\infty}
2^j(\Psi_{j+l}\ast f)(y)\cdot (\Lambda_{j+l}\ast \psi_{j,k})(y)dy\Big|& \lesssim  2^j 2^{3l} \int\limits_{A_{j+l,k}} |(\Psi_{j+l}\ast f)(y)| dy \notag\\
&\lesssim 2^j 2^{3l} 2^{-j-l} \sup\limits_{y \in A_{j+l,k}} |(\Psi_{j+l}\ast f)(y)| \notag\\
&= 2^{2l}  \sup\limits_{y \in A_{j+l,k}} \frac{|(\Psi_{j+l}\ast f)(y)|(1+2^{j+l}|x-y|)^a}{(1+2^{j+l}|x-y|)^a} \notag\\
&\lesssim  2^{2l} P_{2^{j+l},a}(\Psi_{j+l}\ast f)(x). \label{lneg}
\end{align}
From (\ref{f5}), (\ref{lpos}) and (\ref{lneg}) we conclude
$$
|\langle f,2^j \psi_{j,k}\rangle| \lesssim \sum\limits_{l \in \zz} 2^{\beta l} P_{2^{j+l},a}(\Psi_{j+l}\ast f)(x),
$$	
where $x \in I_{j,k}$ and $\beta=a-2$ if $l\geq 0$ and $\beta=2$ if $l<0$.
Then for $x \in \re$ we have
\begin{equation}\label{ineq-coef}
\sum_{k \in \zz} |\langle f,2^j \psi_{j,k}\rangle| \chi_{I_{j,k}}(x) \lesssim \sum\limits_{l \in \zz} 2^{\beta l} P_{2^{j+l},a}(\Psi_{j+l}\ast f)(x).
\end{equation}
 Further we prove part (i).  By using $u$-triangle inequality with $u=\min\{1,p,\theta\}$ we may write
  \begin{align*}
   \|\langle f,2^j \psi_{j,k}\rangle\|_{b_{p,\theta}^{r}}&= \Big(\sum\limits_{j \in \N_{-1}} 2^{\theta r j} \Big\|\sum\limits_{k \in \zz} |\langle f,2^j \psi_{j,k}\rangle| \chi_{j,k} \Big\|_p^\theta \Big)^{1/\theta} \\
   &\lesssim   \Big(\sum\limits_{j \in \N_{-1}} 2^{\theta r j} \Big\|\sum\limits_{l \in \zz} 2^{\beta l} P_{2^{j+l},a}(\Psi_{j+l}\ast f) \Big\|_p^\theta \Big)^{1/\theta} \\
   & \leq \Bigg( \sum\limits_{l \in \zz} 2^{\beta l u} \Big(\sum\limits_{j \in \N_{-1}} 2^{\theta r j} \| P_{2^{j+l},a}(\Psi_{j+l}\ast f)\|_p^\theta  \Big)^{u/\theta} \Bigg)^{1/u}
     \end{align*}
     Using Lemma \ref{peetre-ineq1} we have for  $a>1/p$
  \begin{align*}
     \|\langle f,2^j \psi_{j,k}\rangle\|_{b_{p,\theta}^{r}}
   & \lesssim \Bigg(\sum\limits_{l \in \zz} 2^{\beta l u} 2^{-lru} \Big(\sum\limits_{j \in \N_{-1}} 2^{\theta r (j+l)} \|\Psi_{j+l}\ast f\|_p^\theta  \Big)^{u/\theta} \Bigg)^{1/u}\\
   & \asymp \Bigg( \sum\limits_{l \in \zz} 2^{(\beta-r) l u} \Bigg)^{1/u} \|f\|_{B_{p,\theta}^r}.
  \end{align*}
The series $\sum\limits_{l \in \zz} 2^{(\beta-r) l u} $  is convergent when $1/p-2<a-2<r<2$.

Now we prove (ii). 	We use inequality (\ref{ineq-coef}), $u$-triangle inequality and Lemma \ref{peetre-ineq2}. Then
\begin{align*}
\|\langle f,2^j \psi_{j,k}\rangle\|_{f_{p,\theta}^{r}}&= \Big\| \Big(\sum\limits_{j \in \N_{-1}} 2^{\theta r j} \Big|\sum\limits_{k \in \zz} \langle f,2^j \psi_{j,k}\rangle \chi_{j,k} \Big|^\theta \Big)^{1/\theta}\Big\|_p\\
&\lesssim \Bigg(\sum \limits_{l \in \zz} 2^{\beta l u} 2^{-lru} \Big\|\Big(\sum\limits_{j \in \N_{-1}} 2^{\theta r (j+l)} |\Psi_{j+l}*f|^\theta   \Big)^{1/\theta}\Big\|_p^u \Bigg)^{1/u}\\
&\asymp \Big( \sum \limits_{l \in \zz} 2^{\beta l u} 2^{-lr u} \Big)^{1/u} \|f\|_{F_{p,\theta}^r}
\end{align*}
holds 	with $a>\max\{1/p,1/\theta\}$. On viewpoint of choice of the parameter $r$ the series $\sum \limits_{l \in \zz} 2^{\beta l u} 2^{-lru}$ is convergent.
\end{proof}	

\textbf{Proof of Theorem \ref{charact3-psi}.}
To get (\ref{norm-equi-psi}) we use (\ref{ch1-psi}) and (\ref{ch-psi-2}). The equivalence of norms (\ref{norm-equi-f-psi}) follows from (\ref{ch1-psi-f}) and (\ref{ch-psi-2-f}).  The rest of the proof can be obtained by using the same technique as in Theorem 4.1 (Step 4) \cite{TV2019} and Remark \ref{dual-pair}.
$\blacksquare$

\section{Extension to higher order Faber splines}

In this section we describe the main ideas of extension of results of Sections 2-5 for higher order splines. We start from explicit representation of a dual Chui-Wang wavelets $\psi^*_m$ for $m \in \N$ and $m\geq 3$. In this case we have similar formula
\begin{equation}\label{dual-wav}
\psi_m^*(x)=\sum\limits_{n\in \zz} a_n^{(m)} \psi_m(x-n).
\end{equation}
Further we give a definition of coefficients $a_l^{(m)}$. We consider the following polynomial
\begin{equation}\label{recip-pol}
t^{(m)}(z):=\sum\limits_{n=0}^{4(m-1)} d_n^{(m)}z^n,
\end{equation}
were $d_n^{(m)}=\langle\psi_m(\cdot+n-2(m-1)),\psi_m\rangle$. Since the Chui-Wang wavelet is compactly supported the coefficients  $d_n^{(m)}$ create finite sequence in “increasing symmetric"\ form, i.e.
$$
|d^{(m)}_{2(m-1)}|>|d^{(m)}_{2(m-1)-1}|=|d^{(m)}_{2(m-1)+1}|>|d^{(m)}_{2(m-1)-2}|=|d^{(m)}_{2(m-1)+2}|>...>|d^{(m)}_{0}|=|d^{(m)}_{4(m-1)}|.
$$
The polynomial (\ref{recip-pol}) defined in this way is called self-reciprocal or palindromic. The roots of this polynomial form reciprocal pairs $(\lambda,\frac{1}{\lambda})$ (see, for example, \cite{Harris2012}).

\begin{figure}[h!]
	\begin{minipage}[h]{0.45\linewidth}
		\center{\includegraphics[width=1\linewidth]{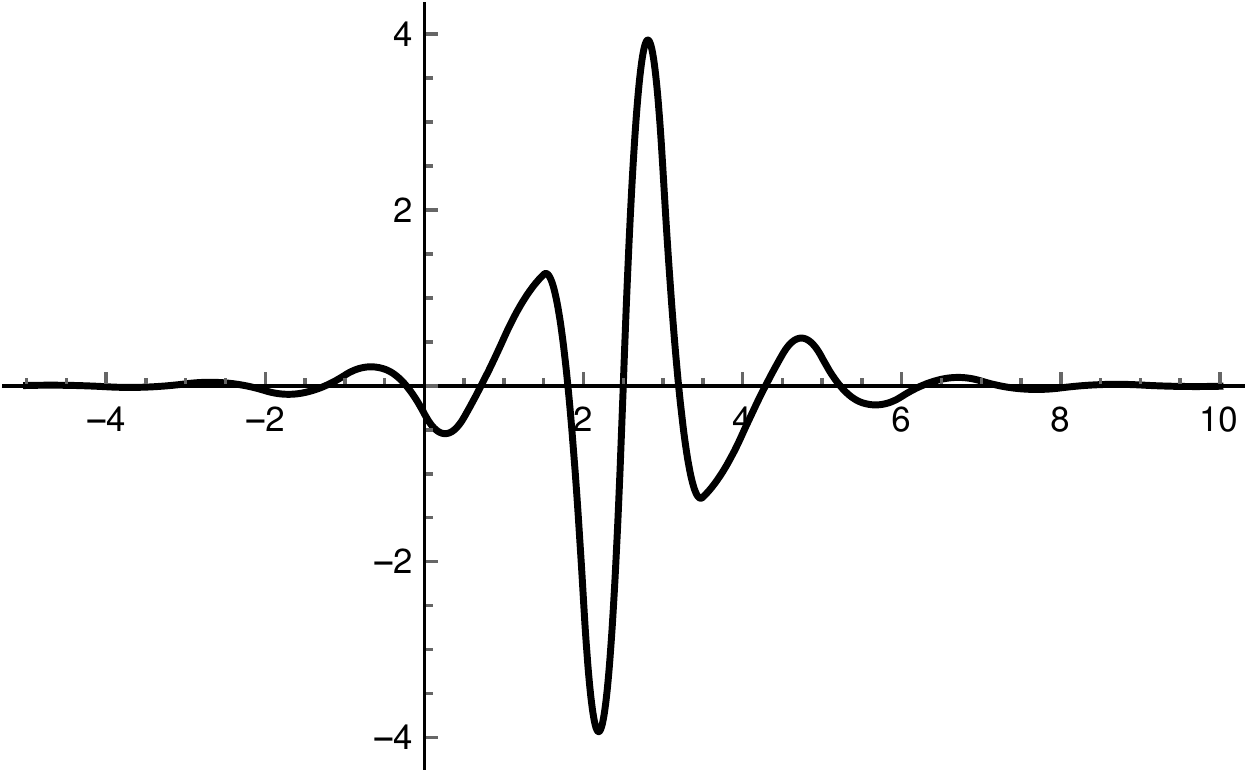}}
	\end{minipage}
	\hfill
	\begin{minipage}[h]{0.45\linewidth}
		\center{\includegraphics[width=1\linewidth]{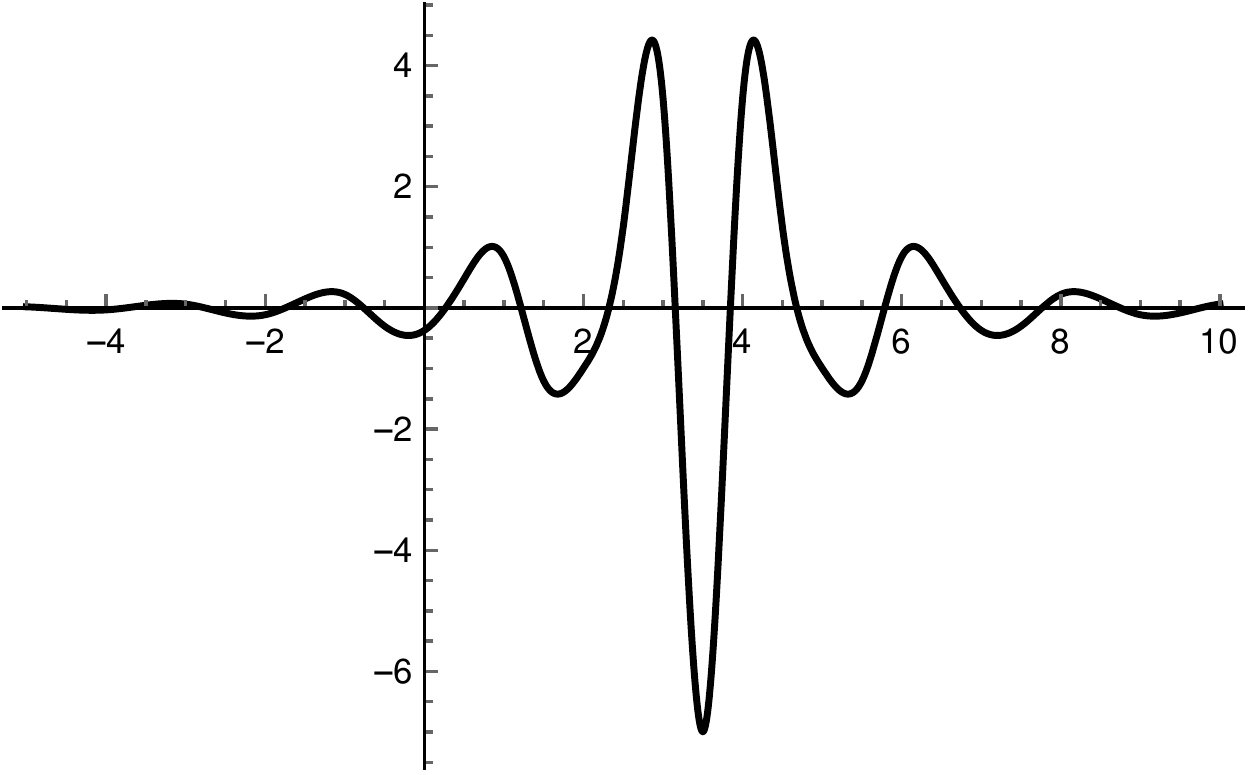}}
	\end{minipage}
	\caption{ Dual wavelets $\psi_3^*$ (left) and $\psi_4^*$ (right).}
	\label{example3-1b}
\end{figure}

 Let $z_0, z_1,..., z_{4(m-1)-1} \in \re$ be the roots of polynomial (\ref{recip-pol}). Let us first show that we never have the case $|z_i|=1$. The system $(\psi_m(\cdot-k))_{k\in \zz}$ is a  Riesz sequence. According to Proposition 2.8 \cite{Woj-book} it means that
\begin{equation}\label{ab}
 A^2\leq \sum\limits_{l\in \zz} |\mathcal{F}\psi_m(\xi +2\pi l)|^2\leq B^2
\end{equation}
 for some $A,B>0$.  From the Poisson summation formula we get
$$
\sum\limits_{l\in \zz} |\mathcal{F}\psi_m(\xi +2\pi l)|^2= \frac{1}{\sqrt{2\pi}} \sum \limits_{n\in \zz} \mathcal{F}\left(\mathcal{F}\psi_m \cdotp \overline{\mathcal{F}\psi_m}\right)(n) \mathrm{e}^{i n \xi}=\frac{1}{\sqrt{2\pi}} \sum \limits_{n\in \zz} \mathcal{F}\left(\mathcal{F}\left(\, \psi_m \ast \psi_m(-\cdot) \,\right)\right)(n) \mathrm{e}^{i n\xi}.
 $$

By using definitions of the Fourier transform and the inverse Fourier transform it is easy to show that $\mathcal{F}\left(\mathcal{F}(g)\right)(t)=\sqrt{2\pi} g(-t)$. Therefore,
 $$
 \sum\limits_{l\in \zz} |\mathcal{F}\psi_m(\xi +2\pi l)|^2= \sum\limits_{n\in \zz} \langle \psi_m,\psi_m(\cdot-n) \rangle  \mathrm{e}^{i n\xi}=t^{(m)}(\mathrm{e}^{i \xi}).
 $$

From the last equality  and (\ref{ab}) we have that the polynomial (\ref{recip-pol}) is bounded away from zero on the unite circle $|z|=1$, therefore it does not have roots that satisfy $|z_i|=1$.

  We rearrange roots in the following way: let $z_0,...,z_{2(m-1)-1}$ be the roots that are located inside of the unite circle, i.e $|z_i|<1$ for $i=0,...,2(m-1)-1$, and $z_{2(m-1)},...,z_{4(m-1)-1}$ be the roots that are located outside of the unit circle, i.e $|z_i|>1$ for $i=2(m-1),...,4(m-1)-1$. Then for $n\leq 2(m-1)-1$
\begin{equation}\label{an-1}
a_n^{(m)}= \frac{(-1)^{m+1}}{|d_0^{(m)}|}\sum\limits_{i=0,...,2(m-1)-1}\frac{1}{z_i^{n-(2(m-1)-1)} \prod\limits_{j=0,...,2(m-1)-1,j\neq i}(z_i-z_j)}
\end{equation}
and	for $n\geq 2(m-1)-1$
\begin{equation}\label{an-2}
a_n^{(m)}=\frac{(-1)^{m}}{|d_0^{(m)}|}\sum\limits_{i=2(m-1),...,4(m-1)-1}\frac{1}{z_i^{n-(2(m-1)-1)} \prod\limits_{j=2(m-1),...,4(m-1)-1,j\neq i}(z_i-z_j)} .
\end{equation}
Since we have the multiplier  $\frac{1}{z_i^{n-(2(m-1)-1)}}$ we get again exponential decay of coefficients $a_n^{(m)}$ with respect to $n$.
Note that since coefficients $a_n^{(m)}$ are defined by
$$
a_n^{(m)}=\frac{1}{2\pi} \int \limits_{0}^{2\pi} \frac{\mathrm{e}^{-inx}}{t^{(m)}(\mathrm{e}^{ix}) \mathrm{e}^{-i(2m-1)x}} \, dx,
$$
i.e they are Fourier coefficients of the function $\frac{1}{t^{(m)}(x)}$ (see Section 2 for details) which is even due to properties of coefficients $d_n^{(m)}$, we have that $a_n^{(m)}$ are real. At Figure 3 we present plot of dual wavelets for cases $m=3$ and $m=4$. The coefficients $a_n^{(3)}$ and $a_n^{(4)}$ are computed numerically.

Below we present the algorithm for computation of these coefficients for each $m \in \N$ and $m\geq 2$ (see Algorithm 1).

\begin{algorithm}[h]
	\caption{Computation of coefficients $a_n^{(m)}$}\label{algo1}
	\begin{tabular}{p{1.2cm}p{5.0cm}p{8.1cm}}
		Input: &$m \in \N$
	\end{tabular}
	\begin{algorithmic}
		\STATE
		1) Compute coefficients $d_n^{(m)}$ for $n=0,...,4(m-1)$ by
		$$
		d_n^{(m)}=\int\limits_{\re} \psi_m(x+n-2(m-1)) \psi_m(x) \, dx;
		$$
         2) Find the roots
        \begin{equation}\label{equa-roots}
         \sum\limits_{n=0}^{4(m-1)} d_n^{(m)}z^n=0;
         \end{equation}
         3) Divide roots $z_0,...,z_{4(m-1)-1}$ for two sets
         \begin{itemize}
         	\item $|z_i|<1$ for $i=0,...,2(m-1)-1$,
         	\item $|z_i|>1$ for $i=2(m-1),...,4(m-1)-1$;
         \end{itemize}
         4) Compute coefficients $a_n^{(m)}$ by (\ref{an-1}) and (\ref{an-2}).
	\end{algorithmic}
\hspace{-8.5cm}	Output:  Coefficients $a_n^{(m)}$.
\end{algorithm}

Analogical, we write for $N_m^*$ that is dual to $N_m$
  \begin{equation}\label{dual-sp}
N_m^*(x)=\sum\limits_{n\in \zz} b_n^{(m)} N_m(x+m/2-n).
\end{equation}

Further we define coefficients $b_l^{(m)}$. Let $z_0, z_1,..., z_{2(m-1)-1} \in \re$ be the roots of the palindromic polynomial of the order $2(m-1)$
$$
\sum\limits_{n=0}^{2(m-1)} g_n^{(m)}z^m,
$$
were $g_n^{(m)}=\langle N_m(\cdot+m/2+n-(m-1)),\psi_m(\cdot+m/2)\rangle$. Let again by $z_0,...,z_{m-2}$ we denote the roots that are located inside of the unite circle and by $z_{m-1},...,z_{2(m-1)-1}$ roots that are located outside of the unit circle. Then
$$
b_n^{(m)}= \frac{1}{|g_0^{(m)}|}\sum\limits_{i=0,...,m-2}\frac{1}{z_i^{n-(m-2)} \prod\limits_{j=0,...,m-2,j\neq i}(z_i-z_j)} , \ \ n\leq m-2
$$
and	
$$
b_n^{(m)}=\frac{1}{|g_0^{(m)}|}\sum\limits_{i=m-1,...,2(m-1)-1}\frac{1}{z_i^{n-(m-2)} \prod\limits_{j=m-1,...,2(m-1)-1,j\neq i}(z_i-z_j)} , \ \ n\geq m-2.
$$

\begin{rem}\label{coef-anm}
Note that a method of computation of coefficients $c_n^{(m)}$ from the definition of the cardinal spline function (\ref{card-sp}) is presented in \cite[\S 4]{Chuibook}. By using similar technique one can probably also compute coefficients $a_n^{(m)}$ for a dual representation of Chui-Wang wavelets (\ref{dual-wav}). On the other hand, coefficients $c_n^{(m)}$ coincide with $b_n^{(m)}$ from the representation for dual B-spline $N_m^*$ (\ref{dual-sp}), more precisely $c_n^{(2m)}=b_n^{(m)}$.
\end{rem}

We denote $\psi_{m;-1,k}:=N_m(\cdot+m/2-k)$ and $\psi^*_{m;-1,k}:=N^*_m(\cdot-k)$ and we can write for $f \in L_2(\re)$
\begin{equation}\label{repr-psi-m}
f=\sum \limits_{j\in \N_{-1}}\sum\limits_{k \in \zz}\mu_{m;j,k}(f) \psi^*_{m;j,k},
	\end{equation}
where $\mu_{m;j,k}(f):=\langle f, 2^j \psi_{m;j,k}\rangle$. Now we formulate the analog of Theorem \ref{charact3-psi} for higher order Chui-Wang wavelets.
\begin{satz}\label{charact3-psi-m} \begin{itemize}
		\item [(i)]
		Let $0< p,\theta\leq \infty$, $m \in \N$, $m\geq 2$, $p>1/(2m)$ and $1/p-m<r<\min\{m-1+1/p,m\}$. Then $f \in B_{p,\theta}^r$ can be represented by the series (\ref{repr-psi-m}), which convergent unconditionally in the space $B_{p,\theta}^{r-\varepsilon}$. If $\max\{p,\theta\}<\infty$ we have unconditional convergence in the space $B_{p,\theta}^{r}$.  Moreover, the following norms are equivalent
	$$
		\|\mu(f) \|_{b_{p,\theta}^r} \asymp \|f\|_{B_{p,\theta}^r}.
	$$
		\item [(ii)] 	Let $m \in \N$, $1/(2m)< p,\theta\leq \infty$, $p\neq \infty$, and $\max\{1/p-m,1/\theta-m\}<r<m-1$. Then $f \in F_{p,\theta}^r$ can be represented by the series (\ref{repr-psi-m}), which convergent unconditionally in the space $F_{p,\theta}^{r-\varepsilon}$. If $\theta<\infty$ we have unconditional convergence in the space $F_{p,\theta}^{r}$.  Moreover, the following norms are equivalent
$$
		\|\mu(f) \|_{f_{p,\theta}^r} \asymp \|f\|_{F_{p,\theta}^r}.
$$
	\end{itemize}
\end{satz}

\begin{figure}[h!]
	\center{\includegraphics[width=0.4\linewidth]{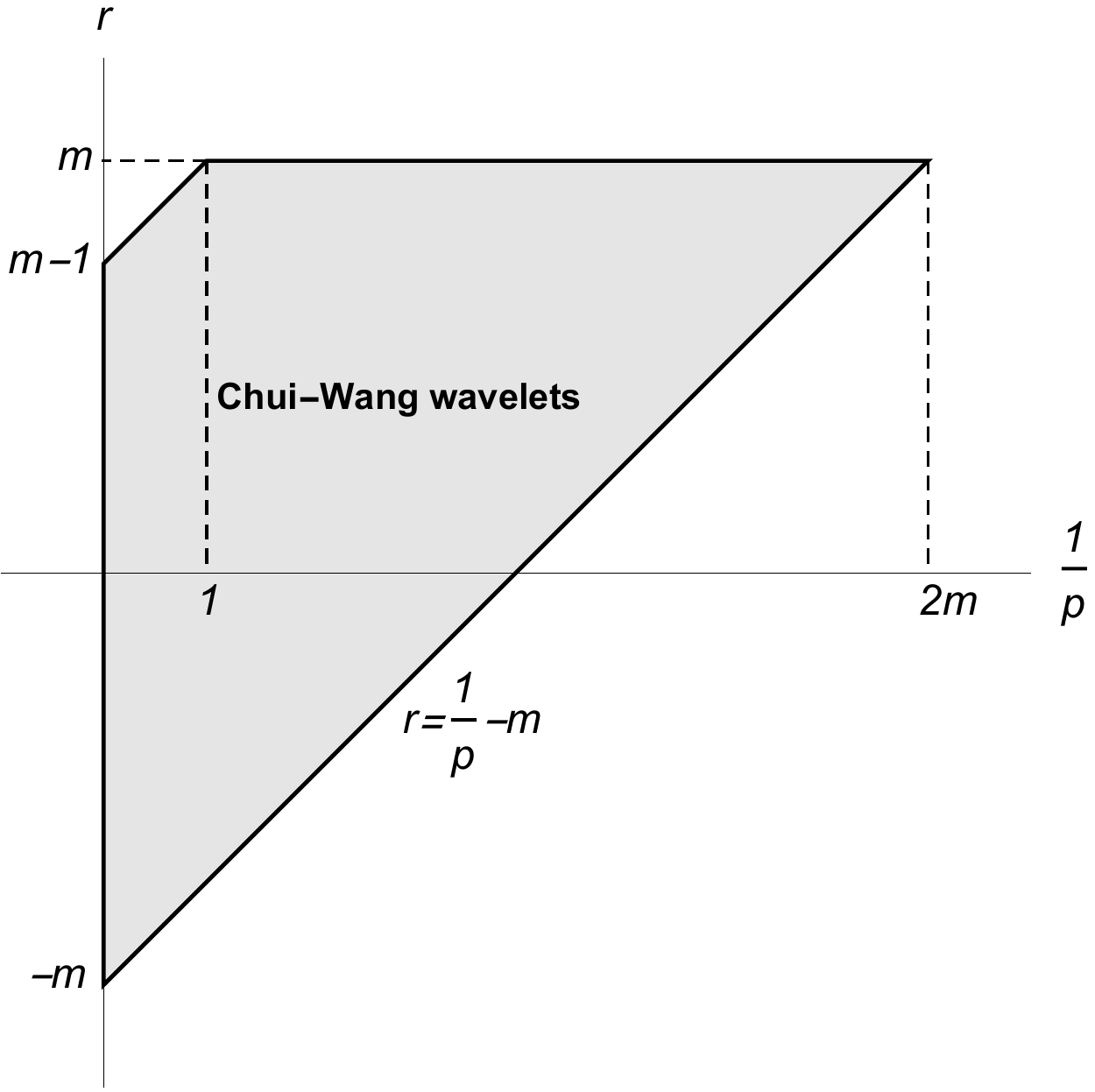}}
	\caption{The range of values of parameter $r$ in Theorem \ref{charact3-psi-m} for B-case.}
	\label{chui-wang-B}
\end{figure}

\begin{rem}
	In Theorem \ref{charact3-psi-m} we consider the situation when $m\geq 2$.  For the Haar basis, i.e. for $m=1$ the corresponding results was obtained in \cite{Tr2010}. See also \cite{Roman2015}--\cite{Roman2016} where the author describes the ``non-tensor'' approach to the multivariate Haar basis.
\end{rem}

By using $m$-th order biortogonal Chui-Wang wavelets and similar technique as in Section 3 we can construct the following $2m$-th order Faber spline basis for $j \in \N_0$:
$$
 s_{2m;j,k}(x):=2^{mj}\int\limits_{-\infty}^x \frac{\psi_{m;j,k}^*(t)}{(m-1)!}(x-t)^{m-1} dt=2^{mj} \sum \limits_{n \in \zz} a_n^{(m)} v_{2m;j,k}
$$
where $v_{2m;j,k}=\int\limits_{-\infty}^x \frac{\psi_{m;j,k}(t)}{(m-1)!}(x-t)^{m-1} dt$ and $s_{2m;-1,k}(x):=L^{2m}(x-k)$. Note that from these formulas for basis functions $s_{2m;j,k}$ we obtain also that they are not compactly supported but due to exponential decay of coefficients $a_n^{(m)}$ "very well localized".

Further we present examples of the construction of Faber splines of 6th order.
\begin{exam}\label{ex-m-3}
For $m=3$ the equation (\ref{equa-roots}) takes the form
$$
1-518z-11072z^2+41734z^3+170110z^4+41734z^5-11072z^6-518z^7+z^8=0.
$$
It has the following roots
\begin{align*}
z_0&= 136+13\sqrt{105}-4\sqrt{2265+221\sqrt{105}},  & z_1=\frac{1}{2}\left(-13-\sqrt{105}+\sqrt{2(135+13\sqrt{105})}\right),\\
z_2&=\dfrac{1}{2}\left(-13+\sqrt{105}+\sqrt{2(135-13\sqrt{105})}\right), & z_3=136-13\sqrt{105}-4\sqrt{2265-221\sqrt{105}}, \\
z_4&= 136+13\sqrt{105}+4\sqrt{2265+221\sqrt{105}},  & z_5=\frac{1}{2}\left(-13-\sqrt{105}-\sqrt{2(135+13\sqrt{105})}\right),\\
z_6&=\dfrac{1}{2}\left(-13+\sqrt{105}-\sqrt{2(135-13\sqrt{105})}\right), &  z_7=136-13\sqrt{105}+4\sqrt{2265-221\sqrt{105}}.
\end{align*}
Then the coefficients $a_n^{(3)}$ are computed numerically by the formulas (\ref{an-1}) and (\ref{an-2})
\begin{align*}
a_0^{(3)}&=12.251,  &a_{\pm 1}^{(3)}&=-3.765,  &a_{\pm 2}^{(3)}&=1.921,  &a_{\pm 3}^{(3)}&=-0.772 ,  &a_{\pm 4}^{(3)}&=0.343,\\
a_{\pm 5}^{(3)}&=-0.145,  &a_{\pm 6}^{(3)}&=6.3\cdot 10^{-2},  &a_{\pm 7}^{(3)}&=-2.7\cdot 10^{-2},  &a_{\pm 8}^{(3)}&=1.1\cdot 10^{-2},  &a_{\pm 9}^{(3)}&=-5.02\cdot 10^{-3},\\
a_{\pm 10}^{(3)}&=2.1\cdot 10^{-3},  &a_{\pm 11}^{(3)}&=-9.3\cdot 10^{-4},  &a_{\pm 12}^{(3)}&=4.01\cdot 10^{-4},  &a_{\pm 13}^{(3)}&=-1.7\cdot 10^{-4},  &a_{\pm 14}^{(3)}&=7.04\cdot 10^{-5},
\end{align*}
and basis functions $s_{6;j,k}$ are defined as follows
$$
s_{6;j,k}(x)= \sum \limits_{n \in \zz} a_n^{(3)}  v_6(2^jx-k-n),
$$
where
$$
v_6(t)=\frac{1}{7200}
\begin{cases}
t^5, & 0\leq t \leq 1/2, \\
1-10t+40t^2-80t^3+80t^4-31t^5, & 1/2<t \leq 1, \\
-236+1175t-2330t^2+2290t^3-1105t^4+206t^5, & 1<t\leq 3/2, \\
6082-19885t+25750t^2-16430t^3+5135t^4-626t^5, & 3/2<t\leq 2, \\
3(-15914+38225t-36270t^2+16950t^3+3895t^4+352t^5), & 2<t\leq 5/2, \\
3(52836-99275t+73730t^2-27050t^3+4905t^4-352t^5), & 5/2<t\leq 3, \\
-250218+383385t-232950t^2+70230t^3-10515t^4+626t^5, &3<t \leq 7/2, \\
186764-240875t+123770t^2-31690t^3+4045t^4-206t^5, & 7/2<t\leq 4, \\
-55924+62485t-27910t^2+6230t^3-695t^4+31t^5, & 4<t\leq 9/2, \\
3125-3125t+1250t^2-250t^3+25t^4-t^5, & 9/2<t\leq 5,\\
0, & \text{otherwise}.
\end{cases}
$$
\end{exam}

Each function $f\in C_0(\re)$ can be expanded in the series
\begin{equation}\label{exp-m}
f=\sum\limits_{j\in\N_{-1}}\sum\limits_{k\in \zz} \lambda_{2m;j,k}(f) s_{2m;j,k},
\end{equation}
in the sense of the uniform norm. The coefficients $\lambda_{2m;j,k}(f)$ are defined as follows $\lambda_{2m;-1,k}(f):=f(k)$ and for $j\in \N_0$
\begin{equation}\label{coef-m}
\lambda_{2m;j,k}(f):=\sum\limits_{l=0}^{2m-2}(-1)^{l}N_{2m}(l+1)\Delta_{2^{-j-1}}^{2m} f\Big( \frac{2k+l}{2^{j+1}}\Big).
\end{equation}

Now we are ready to formulate the theorem about sampling discretization of Besov-Triebel-Lizorkin spaces via higher order Faber splines.
\begin{satz}\label{charact-m} \begin{itemize}
		\item [(i)]
	Let $0< p,\theta\leq \infty$, $m \in \N$, $m\geq 2$, $p>1/(2m)$ and $1/p<r<\min\{2m-1+1/p,2m\}$. Then every compactly supported $f \in B_{p,\theta}^r$ can be represented by the series (\ref{exp-m}), which convergent unconditionally in the space $B_{p,\theta}^{r-\varepsilon}$  for every $\varepsilon>0$. If $\max\{p,\theta\}<\infty$ we have unconditional convergence in the space $B_{p,\theta}^{r}$.  Moreover, the following norms are equivalent
	$$
	\|\lambda(f) \|_{b_{p,\theta}^r} \asymp \|f\|_{B_{p,\theta}^r}.
	$$
		\item [(ii)] 	Let  $1/(2m)< p,\theta\leq \infty$, $p\neq \infty$, and $\max\{1/p,1/\theta\}<r<2m-1$ for $m\in \N$. Then every compactly supported $f \in F_{p,\theta}^r$ can be represented by the series (\ref{exp-m}), which convergent unconditionally in the space $F_{p,\theta}^{r-\varepsilon}$  for every $\varepsilon>0$. If $\theta<\infty$ we have unconditional convergence in the space $F_{p,\theta}^{r}$.  Moreover, the following norms are equivalent
		$$
		\|\lambda(f) \|_{f_{p,\theta}^r} \asymp \|f\|_{F_{p,\theta}^r}.
		$$
		\end{itemize}
\end{satz}

\begin{figure}[h!]
	\center{\includegraphics[width=0.4\linewidth]{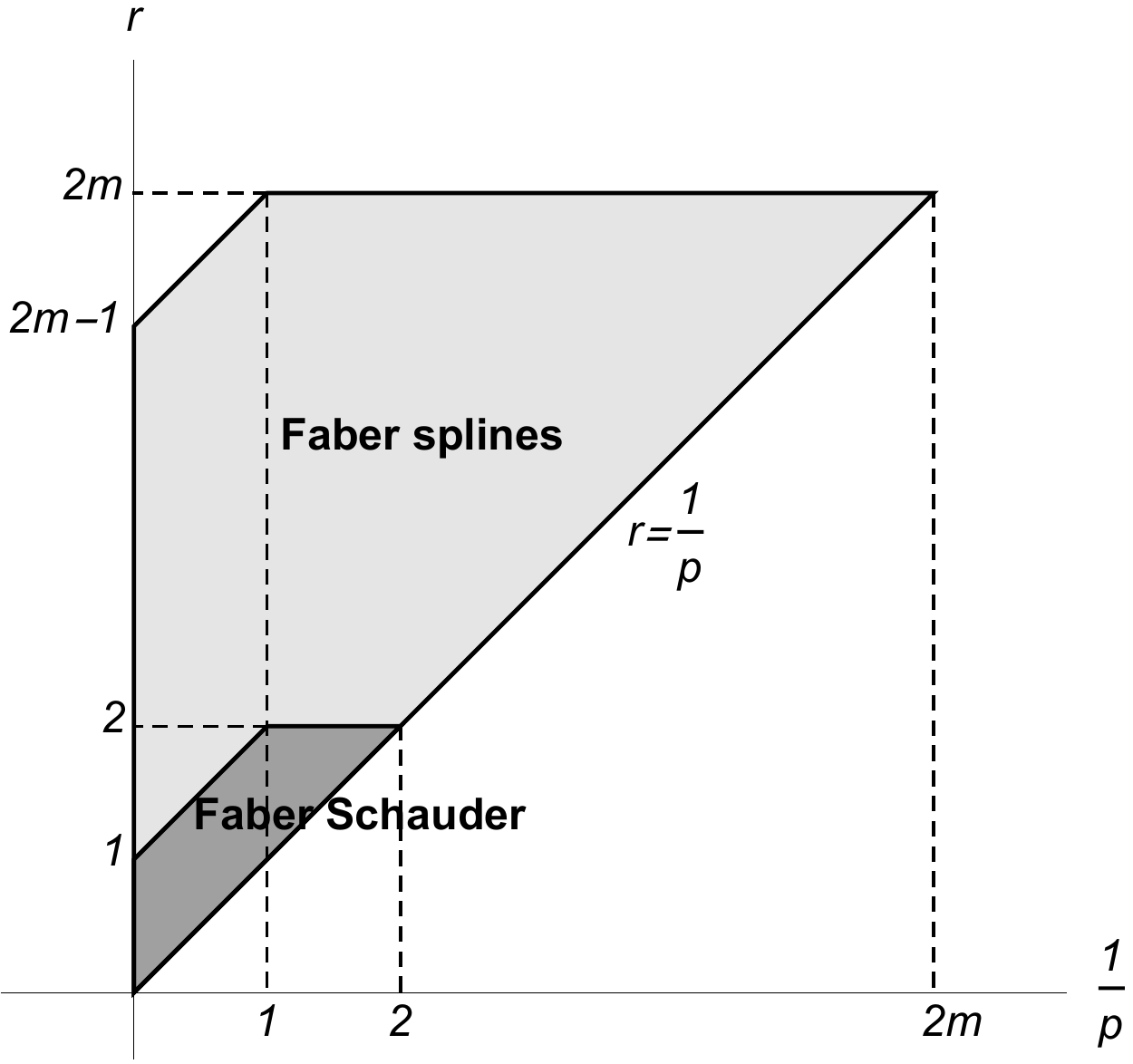}}
	\caption{The range of values of parameter $r$ in Theorem \ref{charact-m} for B-case.}
	\label{faber-spline-B}
\end{figure}

\begin{rem}
In Theorem \ref{charact-m} we consider cases $m\geq 2$. The case $m=1$ corresponds to the Faber Schauder basis and respective characterizations were obtained in \cite{Tr2010}. Note, that for the multivariate situation these results were extended in \cite{Glen2018}.
\end{rem}

\begin{rem}\label{rem-ding}
The operator $S_N$ defined in Lemma \ref{recov2} interpolates a function $f \in C_0(\re)$ at points $k/2^{N}$, $k \in \zz$, i.e. $S_Nf(k/2^N)=f(k/2^N)$ (see Remark \ref{sn-int}). In the papers \cite{DD2009}-\cite{DD2019} the author offers other approach to sampling characterization of Besov spaces based on rather  the quasi-interpolation. This essential difference leads to the fact that the system $s_{2m;j,k}$ constructed in this paper is the ``basis-type''\ system in the sense that this system is linearly independent and further we also prove that $s_{2m;j,k}$ is unconditional basis in Besov-Triebel-Lizorkin spaces while in \cite{DD2009}-\cite{DD2019} the author consider ``frame-type''\ system.  This frame-type system for linear case was also considered in \cite{Oswald2006}.
\end{rem}

\appendix
\section{Definitions and auxiliary statements}

\subsection{Definition of Besov and Triebel-Lizorkin spaces}
First we introduce the concept decomposition the Fourier image called resolution of unity.
\begin{defi}\label{res-unit}
By $\Phi(\re)$ we define a class of systems $\varphi=(\varphi_j)_{j=0}^\infty \subset C_0^\infty(\re)$ that satisfies the following conditions
\begin{itemize}
  \item [(i)] there exists $A>0$ such that $\supp \varphi_0 \subset [-A,A]$;
  \item [(ii)] there are constants $0<B<C$ such that $\supp \varphi_j \subset \{\xi \in \re: B2^j\leq |\xi|\leq C2^j\}$;
  \item [(iii)] for all $\alpha \in \N_0$ there are constants $C_\alpha>0$ such that
  $$
  \sup \limits_{\xi \in \re, \, j \in \N_0} 2^\alpha |D^\alpha \varphi_j(\xi)|\leq C_\alpha <\infty.
  $$
  \item [(iv)] for all $\xi \in \re $
  $$
  \sum\limits_{j=0}^\infty \varphi_j(\xi)=1.
  $$
\end{itemize}
\end{defi}

We define
\begin{equation} \label{deltaj}
\delta_{j}[f](x):=\mathcal{F}^{-1}(\varphi_{j} \mathcal{F} f)(x), \ \ j \in \N_0.
\end{equation}
Then every $f \in S'(\re)$ can be decomposed
\begin{equation} \label{frep}
f=\sum \limits_{j \in \N_0} \delta_{j}[f]
\end{equation}
with convergence in $S'(\re)$.

We define Besov $B_{p,\theta}^r(\re)$ and Triebel-Lizorkin $F_{p,\theta}^r(\re)$ function spaces via the Fourier-analytic building blocks $\delta_{j}[f]$.
\begin{defi}\label{besov}
	Let $\varphi = \{\varphi_j\}_{j=0}^{\infty} \in \Phi(\re)$, $r \in \re$. Then
	\begin{itemize}
		\item [(i)] for $0< p,\theta \leq \infty$ we define
		$$
		B_{p,\theta}^{r}(\re):=\left\{f \in S'(\re): \|f\|_{B_{p,\theta}^{r}(\re)} <\infty \right\},
		$$
		where
		$$
		\|f\|_{B_{p,\theta}^{r}(\re)}:=
		\begin{cases}
		\Big(\sum\limits_{j \in \N_0} 2^{\theta r j} \|\delta_{j}[f]\|_p^\theta   \Big)^{1/\theta}, & 0< \theta <\infty, \\
		\sup \limits_{j \in \N_0}2^{rj} \|\delta_{j}[f]\|_p, & \theta=\infty.
		\end{cases}
		$$
		\item [(ii)] for $0< p,\theta \leq \infty$, $p\neq \infty$ we define
		$$
		F_{p,\theta}^{r}(\re):=\left\{f \in S'(\re): \|f\|_{F_{p,\theta}^{r}(\re)} <\infty \right\},
		$$
		where
		$$
		\|f\|_{F_{p,\theta}^{r}(\re)}:=
		\begin{cases}
	\Big\|	\Big(\sum\limits_{j \in \N_0} 2^{\theta r j} |\delta_{j}[f]|^\theta   \Big)^{1/\theta}\Big\|_p, & 0< \theta <\infty, \\
		\big\|\sup \limits_{j \in \N_0}2^{rj} |\delta_{j}[f]|\big\|_p, & \theta=\infty.
		\end{cases}
		$$
	\end{itemize}
\end{defi}
 Let us give equivalent characterization of Besov-Triebel-Lizorkin spaces via local mean kernels. Let $\Psi_0, \Psi_1 \in S(\re)$ such that for some $\varepsilon>0$: 1) $|\mathcal{F}\Psi_0(\xi)|>0$ for $|\xi|<\varepsilon$; 2) $|\mathcal{F}\Psi_1(\xi)|>0$ for $\varepsilon/2<|\xi|<2\varepsilon$; 3) $D^{\alpha} \mathcal{F}\Psi_1(0)=0$ for $0\leq \alpha <L$. Then for $j\geq 2$
$$
\Psi_j(\xi)=2^{j-1}\Psi_1(2^{j-1}\xi).
$$
From definition it follows that the $L$-th order moment condition hold, i.e. for all $0\leq \alpha<L$
$$
\int\limits_{\re} x^\alpha \Psi_j(x)dx=0.
$$

\begin{satz}\label{local-mean} \cite[Theorem 1.7]{TrIII}
Let $0< p,\theta \leq \infty$ ($p<\infty$ for $F$ case), $\{\Psi_j\}_{j \in \N_0}$ as defined above with $L+1>r$. Then
	$$
	\|f\|_{B_{p,\theta}^{r}(\re)}\asymp
	\begin{cases}
	\Big(\sum\limits_{j \in \N_0} 2^{\theta r j} \|\Psi_{j}*f\|_p^\theta   \Big)^{1/\theta}, & 0< \theta <\infty, \\
	\sup \limits_{j \in \N_0}2^{r j} \|\Psi_{j}*f\|_p, & \theta=\infty,
	\end{cases}
	$$
	and
	$$
\|f\|_{F_{p,\theta}^{r}(\re)}\asymp
\begin{cases}
\Big\|\Big(\sum\limits_{j \in \N_0} 2^{\theta r j} |\Psi_{j}*f|^\theta   \Big)^{1/\theta}\Big\|_p, & 0< \theta <\infty, \\
\Big\|\sup \limits_{j \in \N_0}2^{r j} |\Psi_{j}*f|\Big\|_p, & \theta=\infty,
\end{cases}
$$	
\end{satz}

Further we define discrete function spaces $b_{p,\theta}^r$ and $f_{p,\theta}^r$.
For $j \in \N_{-1}$ and $k \in \zz$ we define the intervals
$$
I_{j,k}:=
\begin{cases}
[2^{-j}k,2^{-j}(k+1)], & j \geq 0, \\
[k-1/2,k+1/2], & j=-1,
\end{cases}
$$
and the corresponding characteristic functions
$$
\chi_{j,k}(x):=
\begin{cases}
1, & x \in I_{j,k},\\
0, & \text{otherwise}.
\end{cases}
$$

\begin{defi}\label{besov-seq}
 Let $r \in \re$ and $0< p,\theta \leq \infty$.	 By $b_{p,\theta}^{r}$ and $f_{p,\theta}^{r}$ ($p<\infty$ for $f$-case) we define the spaces of sequences of coefficients $\lambda=(\lambda_{j,k})_{j \in \N_{-1}, k \in \zz}$ with the finite norms
 $$
 \|\lambda\|_{b_{p,\theta}^{r}}:=
 	\begin{cases}
 	\Big(\sum\limits_{j \in \N_{-1}} 2^{\theta r j} \Big\|\sum\limits_{k \in \zz} \lambda_{j,k} \chi_{j,k} \Big\|_p^\theta \Big)^{1/\theta},& 0< \theta <\infty, \\
 	\sup \limits_{j \in \N_{-1}}2^{r j} \Big\|\sum\limits_{k \in \zz} \lambda_{j,k} \chi_{j,k} \Big\|_p, & \theta=\infty.
 	\end{cases}
 $$
 and
  $$
 \|\lambda\|_{f_{p,\theta}^{r}}:=
 \begin{cases}
 \Big\| \Big(\sum\limits_{j \in \N_{-1}} 2^{\theta r j} \Big|\sum\limits_{k \in \zz} \lambda_{j,k} \chi_{j,k} \Big|^\theta \Big)^{1/\theta}\Big\|_p,& 0< \theta <\infty, \\
\Big\| \sup \limits_{j \in \N_{-1}}2^{r j} \Big| \sum\limits_{k \in \zz} \lambda_{j,k} \chi_{j,k} \Big| \Big\|_p, & \theta=\infty.
 \end{cases}
 $$
 respectively.
\end{defi}

\subsection{Maximal inequalities}
\begin{defi}\label{peetre}
	Let $b>0$ and $a>0$. Then for $f \in L_1(\re)$ with $\mathcal{F}f$ compactly supported  we define the Peetre maximal operator
	$$
	P_{b,a}f(x):=\sup\limits_{y\in \re}\dfrac{|f(y)|}{(1+b|x-y|)^{a}}.
	$$
\end{defi}

\begin{defi}\label{hardy}
For a locally integrable function $f: \re\rightarrow \mathbb{C}$ we define the Hardy-Littlewood maximal operator defined by
$$
(Mf)(x)=\sup\limits_{x \in Q}\frac{1}{Q}\int\limits_{Q}|f(y)| dy, \ \ x \in \re,
$$
where the supremum is taken over all segments that contain $x$.
\end{defi}

\begin{lem}\label{max-ineq}
For $1<p<\infty$ and $1<q\leq \infty$ there exists a constant $c>0$ such that
$$
\Big\|\Big(\sum\limits_{l \in I}|Mf_l|^q \Big)^{1/q} \Big\|_p\leq c \Big\|\Big(\sum\limits_{l \in I}|f_l|^q \Big)^{1/q} \Big\|_p
$$
holds for all sequences $\{f_l\}_{l \in I}$ of locally Lebesgue integrable functions on $\re$.
\end{lem}

\begin{lem}\label{peetre-ineq} \cite[Lemma 3.3.1]{TDiff06}
	Let $a,b>0$, $m\in \N$, $h \in \re$ and $f\in L_1(\re)$ with $\supp \mathcal{F}f\subset [-b,b]$. Then there exists a constant $C>0$ independent of $f,b$ and $h$ such that
	$$
	|\Delta_h^m f(x)|\leq C \min\{1,|bh|^m\} \max\{1,|bh|^a\} P_{b,a}f(x)
	$$
	holds.
\end{lem}

\begin{lem}\label{peetre-ineq1}\cite[1.6.4]{SchTrieb87}
Let $0< p \leq \infty$,  $b>0$ and $a>1/p$. For a bandlimited function $f \in L_1(\re)$ with $\supp \mathcal{F}f\subset [-b,b]$	the following inequality holds
$$
\|P_{b,a}f\|_p	\leq C \|f\|_p,
$$	
where $C$ is some constant independent on $f$ and $b$.
\end{lem}

\begin{lem}\label{peetre-ineq2}\cite[1.6.4]{SchTrieb87}
	Let $0< p,\theta \leq \infty$, $p\neq \infty$, $b^l>0$ for $l\in I$ and $a>\max\{1/p,1/\theta\}$. There exists a constant such that for all systems of functions $\{f_l\}_{l \in I}$ with $\supp \mathcal{F}f_l\subset [-b^l,b^l]$ the following inequality holds
	$$
	\Big\|\Big( \sum\limits_{l \in I}|P_{b^l,a}f_l|^\theta \Big)^{1/\theta} \Big\|_p\leq C \Big\|\Big( \sum\limits_{l \in I}|f_l|^\theta \Big)^{1/\theta} \Big\|_p.
	$$	
\end{lem}

\subsection{Convolution inequalities}
\begin{lem} \label{convol}
	Let $j \in \N_0$, $k,l \in \zz$ and $j+l\geq -1$. Then for the local means $\Psi_j$ with $\supp \Psi_0 \subset [-1/2,1/2]$ and $\supp \Psi_j\subset [-2^{-j},2^{-j}]$ with finitely many  vanishing moments of order $L$ the following estimates hold
	\begin{equation}\label{conv-b-1}
	|\Psi_j * s_{j+l,k}(x)|\leq C 2^{-\alpha|l|} \sum \limits_{n \in \zz} |a_n| \chi_{A_{j+l,k+n}}(x)
	\end{equation}
	and
	\begin{equation}\label{conv-b-2}
	|\Psi_j * s_{j+l,k}(x)|\leq C_R 2^{-\alpha|l|} (1+2^{\min\{j,j+l\}}|x-x_{j+l,k}|)^{-R},
	\end{equation}
	where $\alpha=1$ if $l\geq 0$, $\alpha=3$ if $l<0$. Coefficients $a_n$, $n\in \zz$, are defined by (\ref{dualwavcoef}) for $j+l\geq 0$ and as $a_n=(-1)^n \sqrt{3} (2-\sqrt{3})^{|n|}$ for $j+l=-1$.  For the set $A_{j+l,n+k}$ we have $A_{j+l,n+k}\subset \cup_{|u-k|\lesssim 2^{l_+}} I_{j+l_+,u+n}$.
\end{lem}
\begin{proof}
	Let us first consider the convolution $\Psi_j * v_{j+l,k}$ for $v_{j+l,k}=v(2^{j+l} \cdot +k)$ where $v$ is defined in Lemma \ref{recov2} for $j+l\geq 0$. If $j+l=-1$ we take $v=N_4$ (or if to be strict $N_4(\cdot+2)$).
	
	Since $\supp \Psi_j \subseteq [-2^{-j},2^{-j}]$ and $\supp v_{j+l,k}\subseteq [2^{-j-l}k,2^{-j-l}(k+4)]$
	we can write the following necessary conditions
	$$
	|x-y| \lesssim  2^{-j} \ \ \ \text{and} \ \ \ |x_{j+l,k}-x| \lesssim 2^{-j-l}
	$$
	for the non vanishing of the integral $\Psi_j * v_{j+l,k}(x)$.
	From the triangle inequality we have
	$$
	|x_{j+l,k}-x| \leq |x_{j+l,k}-x| + |x-y| \lesssim \max\{2^{-j}, 2^{-j-l}\}.
	$$
	Therefore, if we denote
	$A_{j+l,k}:=\{x: \, |x_{j+l,k}-x| \lesssim \max\{2^{-j}, 2^{-j-l}\}\}$
	we can write
	$$
	|\Psi_j * v_{j+l,k}(x)|=|\Psi_j * v_{j+l,k}(x)| \chi_{A_{j+l,k}}(x).
	$$
	
	Let us consider the case $j>0$ and $l<0$. In this case the support of $v_{j+l,k}$ is larger than the support of $\Psi_j$. Since $\Psi_j$ has vanishing moments of order 4 and since $v_{j+l,k}$ is piecewise cubic then the integral $\Psi_j * v_{j+l,k}(x)$ is not vanishing in the union of not more then seven intervals centered at nodes of function $v$ with length $\sim 2^j$. For this set we also use notation $A_{j+l,k}\subseteq \cup_{|u-k|\leq c} I_{j,u}$.
	Making change of variables we get
	\begin{align}
	|\Psi_j * v_{j+l,k}(x)|&=\Big| 2^{j-1} \int\limits_{\re} \Psi_1(2^{j-1}(x-y)) v(2^{j+l}y-k) dy \Big| \chi_{A_{j+l,k}}(x) \notag\\
	&=\Big| 2^{-l-1} \int\limits_{\re} \Psi_1(2^{-l-1}(2^{j+l}x-y)) v(y-k) dy \Big| \chi_{A_{j+l,k}}(x) \notag\\
	&=|\Psi_{-l} * v_{0,k}(2^{j+l}x)| \chi_{A_{j+l,k}}(x). \label{negl}
	\end{align}
	Since $v\in B_{\infty,\infty}^3$ by using characterization of Besov spaces via local means Theorem \ref{local-mean} we can proceed the estimate (\ref{negl})
	\begin{align}
	|\Psi_j * v_{j+l,k}(x)|&=2^{3l} |2^{-3l}\Psi_{-l} * v_{0,k}(2^{j+l}x)|\chi_{A_{j+l,k}}(x) \notag\\
	&\leq 2^{3l} \|v_{0,k}\|_{B_{\infty,\infty}^3} \chi_{A_{j+l,k}}(x) \lesssim 2^{3l} \chi_{A_{j+l,k}}(x). \label{neg2}
	\end{align}
	
	For $l\geq 0$ we have
	\begin{align}
	|\Psi_j * v_{j+l,k}(x)|&\leq 2^{j-1} \int\limits_{\re} |\Psi_1(2^{j-1}(x-y))| \, |v_{j+l,k}(y)| dy \, \chi_{A_{j+l,k}}(x) \notag\\
	&\leq2^{j-1} \|\Psi_1\|_\infty \|v_{j+l,k}\|_\infty \int\limits_{2^{-j-l}k}^{2^{-j-l}(k+4)} 1 dy \, \chi_{A_{j+l,k}}(x) \notag\\
	& \lesssim 2^{-l} \, \chi_{A_{j+l,k}}(x). \label{pos}
	\end{align}
	In this case the set $A_{j+l,k}$ is written as $A_{j+l,k}=[2^{-j-l}(k-c2^l),2^{-j-l}(k+c2^l)]$. Therefore, it is easy to see that the following inclusion  $A_{j+l,k} \subset \bigcup_{|u-k|\lesssim 2^{l}} I_{j+l,u}$ takes place.
	
	If $j=0$ we use arguments as in (\ref{pos}). We have
		\begin{equation}\label{pos1}
	|\Psi_0 * v_{l,k}(x)| \leq \|\Psi_0\|_\infty \|v_{l,k}\|_\infty \int\limits_{2^{-l}k}^{2^{-l}(k+4)} 1 dy \, \chi_{A_{l,k}}(x)  \lesssim 2^{-l} \, \chi_{A_{l,k}}(x).
		\end{equation}
	We use the inequality (\ref{pos1}) for $l\geq 0$. If $l=-1$ we write it in the following way
	\begin{equation}\label{neg3}
	|\Psi_0 * v_{-1,k}(x)| \leq C \chi_{A_{-1,k}}(x)=C_{1} 2^{-3} \chi_{A_{-1,k}}(x),
	\end{equation}
	where $\chi_{A_{-1,k}}(x) \subset \bigcup_{|u-k|\lesssim c} I_{c_1,u}$. So we get (\ref{neg2}) for $j=0$ and $l=-1$.

	From (\ref{neg2}), (\ref{pos}), (\ref{pos1}) and (\ref{neg3}) we have that for the functions  $v_{j+l,k}$ the following inequality holds
	\begin{equation}\label{bineq}
	|\Psi_j * v_{j+l,k}(x)|\leq C 2^{-\alpha |l|}\chi_{A_{j+l,k}}(x),
	\end{equation}
	where $\alpha=1$ if $l\geq 0$, $\alpha=3$ if $l<0$ and $A_{j+l,k} \subset \bigcup_{|u-k|\lesssim 2^{l_+}} I_{j+l_+,u}$.
	
	Now we prove the inequality (\ref{conv-b-1}) for functions $s_{j+l,k}$ defined by (\ref{bspline}) and (\ref{b2}).  If $j+l\geq 0$ we consider $a_n$ defined as in (\ref{dualwavcoef}). For $j+l=-1$ we put $a_n=(-1)^n \sqrt{3} (2-\sqrt{3})^{|n|}$. By using (\ref{bineq}) we get
	\begin{align*}
	|\Psi_j * s_{j+l,k}(x)|& = \Big| \sum\limits_{n\in \zz} a_n \int\limits_{\re} \Psi_j(x-y) v(2^{j+l}y-k-n) dy \Big| \\
	& = \Big| \sum\limits_{m\in \zz} a_{m-k} \int\limits_{\re} \Psi_j(x-y) v(2^{j+l}y-m) dy \Big|\\
	& \leq \sum\limits_{m\in \zz} |a_{m-k}| |\Psi_j * v_{j+l,m}(x)|\\
	& \leq C 2^{-\alpha |l|} \sum\limits_{m\in \zz} |a_{m-k}| \chi_{A_{j+l,m}}(x)\\
	& = C 2^{-\alpha |l|} \sum\limits_{n\in \zz} |a_{n}| \chi_{A_{j+l,n+k}}(x).
	\end{align*}
	For the set $A_{j+l,n+k}$ we have $A_{j+l,n+k}\subset \bigcup\limits_{|u-k|\lesssim 2^{l_+}} I_{j+l_+,u+n}$.
	
	Now we prove the inequality (\ref{conv-b-2}). From (\ref{bineq}) we have that for each $R>0$ there is a constant $C_R$ such that
	$$
	\chi_{A_{j+l,k}}(x)\leq C_R(1+2^{\min\{j,j+l\}}|x-x_{j+l,k}|)^{-R}.
	$$
	Therefore,
	$$
	|\Psi_j * v_{j+l,k}(x)|\leq C 2^{-\alpha |l|}C_R(1+2^{\min\{j,j+l\}}|x-x_{j+l,k}|)^{-R}.
	$$
	By using this inequality we can write
	\begin{align*}
	|\Psi_j * s_{j+l,k}(x)|&\leq \sum\limits_{m\in \zz} |a_{m-k}| |\Psi_j * v_{j+l,m}(x)|\\
	& \leq C_R 2^{-\alpha |l|} \sum\limits_{m\in \zz} |a_{m-k}| (1+2^{\min\{j,j+l\}}|x-x_{j+l,m}|)^{-R} \\
	&= C_R 2^{-\alpha |l|} \sum\limits_{n\in \zz} |a_{n}| (1+2^{\min\{j,j+l\}}|x-x_{j+l,n+k}|)^{-R}\\
	&=C_R 2^{-\alpha |l|} (1+2^{\min\{j,j+l\}}|x-x_{j+l,k}|)^{-R} \sum\limits_{n\in \zz}  \frac{|a_{n}|(1+2^{\min\{j,j+l\}}|x-2^{-(j+l)}k|)^{R}}{(1+2^{\min\{j,j+l\}}|x-2^{-(j+l)}(k+n)|)^{R}} \\
	&\leq C_R 2^{-\alpha |l|} (1+2^{\min\{j,j+l\}}|x-x_{j+l,k}|)^{-R} \sum\limits_{n\in \zz} |a_{n}| (1+|n|)^R.
	\end{align*}
	From the viewpoint on definition of coefficients $a_n$ we get convergence of the series $\sum\limits_{n\in \zz} |a_{n}| (1+|n|)^R$ that implies  (\ref{conv-b-2}).
\end{proof}

\begin{lem} \label{convol-chw}
	Let $j \in \N_0$, $k,l \in \zz$ and $j+l\geq -1$. Then for the local means $\Psi_j$ with $\supp \Psi_j\subset [-2^{-j},2^{-j}]$ with finitely many  vanishing moments of order $L$ the following estimates hold
	\begin{equation}\label{conv-psi-1}
	|\Psi_j * \psi^*_{j+l,k}(x)|\leq C 2^{-\alpha|l|} \sum \limits_{n \in \zz} |a_n| \chi_{A_{j+l,k+n}}(x)
	\end{equation}
	and
	\begin{equation}\label{conv-psi-2}
	|\Psi_j * \psi^*_{j+l,k}(x)|\leq C_R 2^{-\alpha|l|} (1+2^{\min\{j,j+l\}}|x-x_{j+l,k}|)^{-R},
	\end{equation}
	where $\alpha=3$ if $l\geq 0$, $\alpha=1$ if $l<0$. Coefficients $a_n$, $n\in \zz$, are defined by (\ref{dualwavcoef}) for $j+l\geq 0$ and as $a_n=(-1)^n \sqrt{3} (2-\sqrt{3})^{|n|}$ for $j+l=-1$.  For the set $A_{j+l,n+k}$ we have $A_{j+l,n+k}\subset \bigcup\limits_{|u-k|\lesssim 2^{l_+}} I_{j+l_+,u+n}$.
\end{lem}
\begin{proof}
	For the case $l<0$ we use the same arguments as in Lemma \ref{convol} and the fact that $\psi \in B^1_{\infty,\infty}(\re)$. This yields $|(\Psi_j\ast \psi_{j+l,k})(x)| \lesssim 2^{l}$ for $l < 0$. For $l\geq0$ we use the moment condition of $\psi$. We subtract the Taylor polynomial of order two to get
\begin{equation*}
   \begin{split}
	|(\Psi_j\ast \psi_{j+l,k})(x)| &= \Big|\int_{-\infty}^{\infty} \Psi_j(x-y)\psi_{j+l,k}(y)\,dy\Big|\\
	&= \Big|\int_{-\infty}^{\infty} (\Psi_j(x-y)-\Psi_{j}(x-x_{j+l,k})\\
	&~~~~~~~~~~~~~~~~~-(\Psi_{j})'(x-x_{j+l,k})(y-x_{j+l,k}))\psi_{j+l,k}(y)\,dy\Big|\\
	&\leq\int_{-\infty}^{\infty} \int\limits_{|t-x_{j+l,k}|\lesssim |y-x_{j+l,k}|}
	|(\Psi_j)''(t)|t-x_{j+l,k}|\,dt \,\psi_{j+l,k}(y)\,dy\\
	&\lesssim 2^{3j}2^{-3(j+l)} = 2^{-3l}\,.
   \end{split}
\end{equation*}
Inequality (\ref{conv-psi-2}) can be proven in a similar way as the inequality (\ref{conv-b-2}).
\end{proof}

\section*{Acknowledgment} The authors would like to thank Glenn Byrenheid, Ding D\~{u}ng,  Gustavo Garrig\'{o}s, Peter Oswald, Martin Sch\"{a}fer,  Andreas Seeger, Winfried Sickel and Hans Triebel for several fruitful discussions on the topic.

\end{document}